\documentclass[a4paper, 12pt]{amsart}
\usepackage{amsmath,amssymb,amsthm}

\usepackage{amscd}
\usepackage{array,enumerate}
\newcommand{\ip}[1]{\mathopen{\langle}#1\mathclose{\rangle}_A}
\newtheorem{Thm}{Theorem}[section]
\newtheorem{Prop}[Thm]{Proposition}
\newtheorem{Lem}[Thm]{Lemma}

\theoremstyle{definition}
\newtheorem{Rem}[Thm]{Remark}

\newtheorem{Exm}[Thm]{Example}

\newcommand{\Cs}{C$^\ast$}

\newcommand{\T}{\mathcal{T}_{\mathcal{E}_{\fr}}}
\newcommand{\C}{\mathcal{O}_{\mathcal{F}_{\fu}}}

\newcommand{\id}{\mbox{\rm id}}

\newcommand{\IB}{\mathbb B}
\newcommand{\IC}{\mathbb C}
\newcommand{\IF}{\mathbb F}
\newcommand{\IK}{\mathbb K}

\newcommand{\IN}{\mathbb N}
\newcommand{\IR}{\mathbb R}
\newcommand{\IT}{\mathbb T}
\newcommand{\fN}{\mathfrak{N}}
\newcommand{\IZ}{\mathbb Z}
\newcommand{\cB}{\mathcal B}
\newcommand{\cF}{{\mathcal F}_{\mathfrak{u}}}

\newcommand{\cS}{\mathcal S}
\newcommand{\cU}{\mathcal U}
\newcommand{\fx}{\mathfrak x}
\newcommand{\fl}{\mathfrak l}
\newcommand{\n}{\mathfrak n}

\newcommand{\fs}{\mathfrak s}
\newcommand{\fS}{\mathfrak S}
\newcommand{\fu}{\mathfrak u}
\newcommand{\m}{\mathfrak m}
\newcommand{\fk}{\mathfrak k}
\newcommand{\fr}{\mathfrak r}
\newcommand{\fy}{\mathfrak y}
\newcommand{\E}{{\mathcal E}_\mathfrak{r}}
\newcommand{\cE}{\mathcal E}
\newcommand{\cZ}{\mathcal Z}
\newcommand{\cI}{\mathcal I}

\newcommand{\cM}{\mathcal M}

\newcommand{\acts}{\curvearrowright}

\DeclareMathOperator{\spa}{\mathop{span}}

\DeclareMathOperator{\Ped}{Ped}

\DeclareMathOperator{\Cso}{{\rm C}^\ast}

\title[]{Amenable actions on finite simple \Cs-algebras arising from flows on Pimsner algebras}

\author{Yuhei Suzuki}
\subjclass[2020]{Primary~
46L55, Secondary~46L35}
\keywords{Non-commutative amenable actions, stably finite \Cs-algebras, Pimsner algebras.}
\address{Department of Mathematics, Faculty of Science, Hokkaido University,
Kita 10, Nishi 8, Kita-Ku, Sapporo, Hokkaido, 060-0810, Japan}

\dedicatory{In honor of Professor Eberhard Kirchberg (1946--2022).}
\email{yuhei@math.sci.hokudai.ac.jp}
\begin{document}
\maketitle

\begin{abstract}Associated to a family of $G$-$\ast$-endomorphisms on a $G$-\Cs-algebra $A$ satisfying certain minimality conditions,
we give a $G$-\Cs-correspondence $\mathcal{E}$ over $A$ whose Cuntz--Pimsner algebra $\mathcal{O}_\mathcal{E}$ is simple.
For certain quasi-free flows $\gamma$ (commuting with the $G$-action) on $\mathcal{O}_\mathcal{E}$,
we further prove the simplicity of the reduced crossed product $\mathcal{O}_\mathcal{E} \rtimes_\gamma \IR$.
We then classify the KMS weights of $\gamma$.
This in particular gives a sufficient condition for $\mathcal{O}_\mathcal{E}$ and $\mathcal{O}_\mathcal{E}\rtimes_\gamma\IR$ to be stably finite
(and to be stably projectionless).
As the amenability of $G\acts A$ inherits to the induced actions $G\acts \mathcal{O}_\mathcal{E}, \mathcal{O}_\mathcal{E}\rtimes_\gamma\IR$,
this provides a new systematic framework to provide amenable actions on stably finite simple \Cs-algebras. 
\end{abstract}

\tableofcontents

\section{Introduction}
Amenability of group actions on \emph{non-commutative} \Cs-algebras is now a central subject in the classification theory of \Cs-algebras.
Such actions were constructed only recently in \cite{Suzeq} (see also \cite{OS}, \cite{Suzsf} for more sophisticated constructions),
but it should be now clear that amenability of actions, not of the acting groups, is the genuine ingredient for classification of group actions on (simple) \Cs-algebras.
For general theory, backgrounds, and other applications of amenable actions on \Cs-algebras, we refer the reader to \cite{AD}, \cite{BEW2}, \cite{OS}, and references therein.

In the realm of \emph{Kirchberg algebras} \cite{Kir},
the dynamical version of Kirchberg's $\mathcal{O}_2$- and $\mathcal{O}_\infty$-tensor absorption theorems and $\mathcal{O}_2$-embedding theorem are established in \cite{Suz21} for exact countable groups
(see \cite{KP} for the original theorem, \cite{Sza} for the amenable group case, and \cite{GS2} for the non-discrete case),
and more surprisingly and significantly, the dynamical version of the Kirchberg--Phillips classification theorem (\cite{Kir}, \cite{Phi}) for amenable actions
is established by Gabe and Szab\'o \cite{GS1}, \cite{GS2}.
The latter theorem basically settles all open problems on group actions on Kirchberg algebras, in much better forms than previously conjectured.
We refer the reader to \cite{GS1}, \cite{GS2} for details.
For some backgrounds, developments, and history of this subject before the Gabe--Szab\'o classification theorem,
we refer the reader to the introduction of \cite{Sza}.

Now, considering these astonishing achievements, we naturally obtain a new next important goal in classification theory of \Cs-algebras:
Extend the Gabe--Szab\'o classification theorem to \emph{classifiable} simple \Cs-algebras (see e.g., \cite{Win}, \cite{GLN}) \emph{beyond} Kirchberg algebras.
However, as the constructions of \emph{stably finite} simple \Cs-algebras (see \cite{Ell}, \cite{EV}, \cite{GLN} for instance) are typically much more complicated, involved, and delicate than the purely infinite case,
naturally it is more difficult to construct amenable actions (of non-amenable groups) on stably finite simple \Cs-algebras.
Indeed, such actions were discovered only very recently in \cite{Suzsf}.
(We note that the existence of such an action on a \emph{unital} stably finite simple \Cs-algebra still remains unsolved.)
As the construction therein depends on various specific ingredients (the right Laplacian on the Cayley graph, the full Fock space of the regular representation, and a trace-scaling automorphism on a simple \Cs-algebra), it seems difficult to control and determine the invariants (most importantly, the equivariant Kasparov classes \cite{Kas}) of the resulting actions in general.

The purpose of the present article is to give a systematic framework to provide amenable actions on stably finite simple \Cs-algebras.
As it turns out, our framework is particularly effective to produce amenable actions on simple \emph{stably projectionless} \Cs-algebras.
Since the structure of simple stably projectionless \Cs-algebras is relatively simpler than the other stably finite simple \Cs-algebras
(for instance, the order structure on the K$_0$-group is trivial in the former class), the next natural (possibly tractable) target in this theory would be
the classification of amenable actions on simple stably projectionless \Cs-algebras.
A positive evidence was recently established by Nawata \cite{Naw}, in which he proved the dynamical version of the $\mathcal{W}$-tensor absorption theorem
for \emph{all} countable amenable groups.
It should be also worth to note that some dynamical analogues of the $\mathcal{Z}$-tensor absorption theorem
were recently studied for general countable amenable groups by many authors.
For this topic, we refer the reader to \cite{Sat}, \cite{SW}, and references therein.
Some of these approaches would have a potential to be extended to amenable actions of not-necessarily-amenable groups
after suitable improvements and reformulations.
(We however remark that amenable actions of non-amenable groups have
non-trivial obstructions on important invariants, notably on the induce action on the trace space --- see e.g., \cite{AD79}, Proposition 3.5 of \cite{OS}, and Remark 2.1 (3) of \cite{Suzsf} --- and on the equivariant Kasparov class --- see \cite{OS}, Proposition 6.5.
Therefore not all of these results, particularly \cite{Naw}, are directly extended to general amenable actions.)

We next briefly view our new constructions that will be established in this article.
Our main ingredients are (equivariant) \Cs-correspondences and flows on them.
This may be seen as a hybrid of the constructions in \cite{OS} and \cite{Suzsf}.
Recall that in \cite{OS}, we establish a powerful framework to provide amenable actions
on purely infinite simple \Cs-algebras.
More precisely, we have shown that the construction of Pimsner--Kumjian--Meyer (\cite{Pim}, \cite{Kum}, \cite{Mey})
preserves the amenability of the action.
(See Theorems 5.1 and 6.1 of \cite{OS} for the precise statements.)
The construction always gives purely infinite simple \Cs-algebras
and does not change important properties and invariants (notably it preserves the equivariant Kasparov class), whence this gives the desired construction in the purely infinite case.
To compare this construction with our new constructions, let us first review the Pimsner--Kumjian--Meyer construction.
For a $G$-\Cs-algebra $A$,
take a faithful equivariant $\ast$-representation $\pi \colon A \rightarrow \IB(H)$ on a $G$-Hilbert space $H$ with no non-trivial compact elements in the image.
For each such a representation, Meyer defined the $G$-\Cs-correspondence $\cE$ over $A$,
which is equal to $H \otimes A$ as a right Hilbert \Cs-$A$-module,
whose left $A$-action is given by $\pi \otimes 1_A$.
(This is an equivariant version of Kumjian's construction \cite{Kum}.)
Kumjian \cite{Kum} has shown that the Toeplitz--Pimsner algebras of these $\cE$
are purely infinite simple.
Since the Toeplitz--Pimsner construction preserves the (equivariant) Kasparov class (see \cite{Pim} and \cite{Mey}),
this gives a useful construction of amenable actions on purely infinite simple \Cs-algebras.

However, as proved by Kumjian \cite{Kum}, this construction never gives rise to a stably finite \Cs-algebra.
Also, because $A$ acts on the Hilbert space side $H$ of $\cE=H \otimes A$,
it seems hard to provide a flow on the $G$-\Cs-correspondence $\mathcal{E}$
whose induced flow on the Pimsner algebra admits a KMS weight, as the existence of a KMS weight is usually a rather delicate condition.

To settle these difficulties, we introduce a new $G$-\Cs-correspondence,
which is associated to a family $\fr$ of $G$-equivariant $\ast$-endomorphisms on a given $G$-\Cs-algebra.
Informally speaking, the resulting \Cs-algebra can be seen as the crossed product of $\fr$ by $\mathcal{O}_{|\fr|}$; ``$A \rtimes _{\fr} \mathcal{O}_{|\fr|}$''.
(In fact these \Cs-algebras are studied by Cuntz \cite{Cun} for a different purpose, under the additional assumption that $\fr$ consists of mutually commuting $\ast$-automorphisms.) 
A main theorem of the present article shows that
under a mild minimality condition on $\fr$, the resulting Cuntz--Pimsner algebra
and the reduced crossed product \Cs-algebras of certain quasi-free flows
are simple.
Then we establish a classification theorem of KMS weights for certain quasi-free flows.
This gives a useful criterion for the stable finiteness (and for being stably projectionless) of the resulting simple \Cs-algebras.
This sheds new light on the construction of non-commutative amenable actions:
Our new approach clarifies a new importance of the study of non-discrete group actions,
even when one is only interested in discrete groups.
(Such importance has been pointed out in \cite{OS} for compact groups.
But our new results also clarify importance of actions of non-compact groups.)
In particular this indicates that a possibly existing \Cs-algebra version of the Tomita--Takesaki theory suggested in \cite{KK} would play significant roles in this theory.

Our analysis is based on the method in the seminal works of Kishimoto and Kumjian \cite{KK}, \cite{KK2} on quasi-free flows on the Cuntz algebras $\mathcal{O}_n$ ($2 \leq n <\infty$) (see also \cite{Kat} for a partial generalization to $\mathcal{O}_\infty$).
However we note that the canonical isometric elements in (the multiplier algebra of) the Pimsner algebra
non-trivially act on the primitive ideal space of the coefficient algebra in the present situation (which are substantial requirements in our new constructions),
while they are certainly trivial in the Cuntz algebra case.
Therefore we need various modifications and improvements.
We believe that these results are of independent interest, even in the non-equivariant case.

\subsection*{Notations}
Let $A$ denote a \Cs-algebra.

\begin{itemize}
\item We let $\IN:=\{0, 1, 2, \ldots\}$ denote the set of all non-negative integers.
\item $A_+$ denotes the positive cone of $A$.
\item For a subset $X \subset A$, set
\[(X)_1:=\{x\in X: \|x\|\leq 1\},\quad X_+:= X \cap A_+.\]
Clearly these definitions do depend on the inclusion $X \subset A$,
but this will be always clear from the context.
\item For $x, y\in A$ and a positive number $\epsilon>0$,
denote by $x\approx_\epsilon y$ if $\|x-y\|<\epsilon$.
\item
For a right Hilbert \Cs-$A$-module ${\mathcal{E}}$,
denote by $\IB_A({\mathcal{E}})$, $\IK_A(\mathcal{E})$
the \Cs-algebra of all bounded adjointable operators
and the \Cs-algebra of all compact operators on
${\mathcal{E}}$ respectively.
We let $\cU_A({\mathcal{E}})$ denote the unitary group of the \Cs-algebra $\IB_A({\mathcal{E}})$.
\item We usually assume ideals of a \Cs-algebra to be closed and self-adjoint.
However (possibly) non-closed self-adjoint ideals also appear in our analysis.
We refer to such ideals as \emph{algebraic} ideals. 
\item For a subset $S\subset A$, denote by $\cI(S; A)$ the ideal of $A$ generated by $S$.
\item The symbols `$\rtimes$', `$\otimes$' stand for the \Cs-algebra reduced crossed product and the \Cs-algebra minimal tensor product, respectively.
\item Put $\IK:=\IK(\ell^2(\IN))$.
\item $\cM(A)$ denotes the multiplier algebra of $A$.
\item For $a\in \cM(A)_+$,
denote by ${\rm Her}(a; A)$ the hereditary \Cs-subalgebra of $A$ generated by $a$.
That is, ${\rm Her}(a; A)=\overline{aAa}$.

\item A \Cs-algebra inclusion $A \subset B$, or more generally, a completely positive map $\varphi \colon A \rightarrow B$
between \Cs-algebras, is said to be \emph{non-degenerate} if $\varphi$ sends an approximate unit of $A$ to that of $B$.
By Corollary 5.7 of \cite{Lan}, such a map $\varphi$ extends to a strictly continuous completely positive map $\cM(A)\rightarrow \cM(B)$.
We denote the extension by the same symbol $\varphi$.
\end{itemize}

\section{Preliminaries}
Here we collect some basic facts used in our analysis.

\subsection{KMS weights and tracial weights}
For basic facts on KMS weights on \Cs-algebras, we refer the reader to \cite{Thom}.
A \emph{weight} on a \Cs-algebra $A$ is a map
\[\varphi \colon A_+ \rightarrow [0, \infty]\]
satisfying the following conditions
\begin{itemize}
\item $\varphi(0)=0$,
\item $\varphi(a+b)=\varphi(a)+\varphi(b)$ for $a, b\in A_+$,
\item $\varphi(\lambda a)=\lambda\varphi(a)$ for $\lambda \in [0, \infty), a\in A_+$.
\end{itemize}
A weight is said to be \emph{proper} if it is nonzero, lower semi-continuous, and densely defined (that is, $\varphi^{-1}([0, \infty))$ is dense in $A_+$).
Throughout this article, we only consider proper weights.
Therefore we omit the term `proper' unless the omission may cause a confusion.

For a weight $\varphi$ on a \Cs-algebra $A$, put
\[\cM_\varphi:=\spa \varphi^{-1}([0, \infty)).\]
Note that $\cM_\varphi$ is a dense $\ast$-subalgebra of $A$ (\cite{Thom}, Lemma 1.0.1 (e)).
Moreover, since $\varphi^{-1}([0, \infty)) \subset A_+$ is a hereditary subcone, it coincides with $(\cM_\varphi)_+$.
It is not hard to show that
$\varphi|_{(\cM_\varphi)_+}$ uniquely extends to a linear functional on $\cM_\varphi$ (\cite{Thom}, Lemma 1.2.4).
We denote the resulting linear functional $\cM_\varphi \rightarrow \IC$ by the same symbol $\varphi$.

We next recall the KMS condition.
Consider a flow $\alpha\colon \IR \acts A$ on a \Cs-algebra $A$.
An element $a\in A$ is said to be \emph{$\alpha$-entire} if
the orbit map $t\mapsto \alpha_t(a)$ extends to
an $A$-valued entire function
\[\IC \ni z \mapsto \alpha_z(a).\]
For each $n\in \IN \setminus \{0\}$, define a positive contractive map $R_n\colon A \rightarrow A$
to be
\[R_n(a):=\sqrt{\frac{n}{\pi}}\int_\IR e^{-ns^2}\alpha_s(a)ds, \quad a\in A.\]
Observe that for any $a\in A$, $R_n(a)$ is $\alpha$-entire.
Indeed, the function
\[\alpha_z(R_n(a)):=\sqrt{\frac{n}{\pi}}\int_\IR e^{-n(s-z)^2}\alpha_s(a)ds, \quad z\in \IC\]
gives an analytic extension of the orbit map
$t \mapsto \alpha_t(R_n(a))$.
This formula also shows the boundedness of the composite
$\alpha_z \circ R_n$ for all $z\in \IC$.
We put $R_0:=0$ for notational convenience.

By the Identity Theorem of analytic functions, for any $\alpha$-entire elements $a, b\in A$, $n\in \IN$, and $z, w\in \IC$, one has
\[\alpha_z(ab)=\alpha_z(a)\alpha_z(b),\quad \alpha_z(\alpha_w(a))=\alpha_{z+w}(a),\quad \alpha_z(R_n(a))=R_n(\alpha_z(a)), \quad \alpha_z(a^\ast)= \alpha_{\bar{z}}(a)^\ast \]
as the equations are valid when $z\in \IR$ and all terms are analytic on $z\in \IC$.

We set 
\[\mathcal{A}_\alpha:=\{a\in A: a{\rm~is~}\alpha\mathchar`-{\rm entire}\}.\]
Note that $(R_n(a))_{n\in \IN}\subset \mathcal{A}_\alpha$ converges to $a$ for all $a\in A$,
hence $\mathcal{A}_\alpha$ is dense in $A$.
It is clear from formulas in the previous paragraph that $\mathcal{A}_\alpha$ is a $\ast$-subalgebra of $A$ and satisfies $\alpha_z(\mathcal{A}_\alpha)=\mathcal{A}_\alpha$ for all $z\in \IC$.

For an $\alpha$-invariant weight $\varphi$ on $A$, we set
\[\cM_\varphi^\alpha:=\bigcap_{z\in \IC} \alpha_z(\mathcal{A}_\alpha \cap \cM_\varphi).\]
Note that $\cM_{\varphi}^\alpha$ is a $\ast$-subalgebra of $A$.
Moreover it is dense in $A$, since one has $R_n(\cM_{\varphi})\subset\cM_\varphi^\alpha$ for all $n\in \IN$.

For $\beta\in \IR$,
an $\alpha$-invariant weight $\varphi$ on $A$ is said to be an \emph{$\alpha$-$\beta$-KMS weight} (or \emph{satisfy the $\alpha$-$\beta$-KMS condition}) if it satisfies
one of the following two equivalent conditions:
\begin{enumerate}
\item For any $a\in \cM_\varphi^\alpha$, one has
\[\varphi(a^\ast a)=\varphi(\alpha_{-i\beta/2}(a) \alpha_{-i\beta/2}(a)^\ast).\]
\item For any $a, b\in \cM_\varphi^\alpha$, one has
\[\varphi(ab)=\varphi(b\alpha_{i\beta}(a)).\]
\end{enumerate}
We refer the reader to Theorem 3.2.1 of \cite{Thom}
for a proof of the equivalence and a few more equivalent formulations.

A weight $\tau$ on $A$ is said to be \emph{tracial} if it satisfies the tracial condition
\[\tau(a^\ast a)=\tau(a a^\ast)\quad{\rm~for~all~}a\in A.\]
Observe that $\alpha$-$0$-KMS weights are nothing but $\alpha$-invariant tracial weights.
For any tracial weight $\tau$, by the tracial condition,
the dense $\ast$-subalgebra $\cM_\tau\subset A$ in fact forms an algebraic ideal of $A$.
Consequently $\cM_\tau$ contains the \emph{Pedersen ideal} ${\rm Ped}(A)$ of $A$.
(For basic facts on Pedersen ideals, we refer the reader to Section 5.6 of \cite{Pedbook}.)

\begin{Rem}
An approximate unit $(e_n)_{n\in \IN}$ satisfying the condition $e_n e_{n+1} = e_n$ for each $n\in \IN$
makes various arguments on weights much easier.
For any separable \Cs-algebra $A$, one can find such an approximate unit in $\Ped(A)$.
Indeed, take a strictly positive element $a\in (A_+)_1$
and take an increasing sequence $(f_n) \subset C_c((0, 1])_+$
satisfying $\lim_{n \rightarrow \infty} f_n(x)=1$ for all $x\in (0, 1]$ and $f_n f_{n+1}=f_n$ for all $n\in \IN$.
Then the elements $e_n:=f_n(a)$; $n\in \IN$ form the desired approximate unit.
\end{Rem}

Let $A_0$ be a $\ast$-subalgebra of a \Cs-algebra $A$.
We say that a linear functional $\varphi \colon A_0 \rightarrow \IC$ is \emph{positive}
if it is self-adjoint and satisfies $\varphi((A_0)_+)\subset [0, \infty)$.
We remark that this definition does depend on the ambient \Cs-algebra $A$.

The following observation is useful to classify tracial weights on a \Cs-algebra.
\begin{Prop}\label{Prop:trace}
Let $A_0 \subset A$ be a dense $\ast$-subalgebra of a \Cs-algebra $A$
which contains an approximate unit $(e_n)_{n\in \IN} \subset A_0$ of $A$
satisfying $e_n e_{n+1}=e_n$ for each $n\in \IN$.
Define $A_1:=\bigcup_{n\in \IN} e_n A_0 e_n$, which is a dense $\ast$-subalgebra of $A$ containing $(e_n)_{n\in \IN}$.
Then any tracial positive linear functional
$\tau_0 \colon A_1 \rightarrow \IC$
uniquely extends to a tracial weight $\tau$ on $A$.
\end{Prop}
\begin{proof}
We first note that for any positive linear functional $\varphi\colon A_1 \rightarrow \IC$ and any two positive elements $e, f\in (A_1)_+$
satisfying $ef=f$, $\varphi$ is bounded on $eA_1 e$ (in the norm of $A$).
Indeed for any self-adjoint element $a\in e A_1 e$, since $a=faf$, we have
\[ - \|a\| f^2 \leq a \leq\|a\| f^2,\]
and hence $\|\varphi|_{e A_1 e}\| \leq 2\varphi(f^2)$.
Thus the restriction $\varphi|_{e A_1 e}$ extends to a bounded positive linear functional on $\overline{e A_1 e}={\rm Her}(e; A)$.

By the above observation, $\tau_0$
extends to a tracial positive linear functional $\tau_1$ on the increasing union $\bigcup_{n\in \IN}{\rm Her}(e_n; A)\subset A$.
(This is well-defined, thanks to the relations $e_n e_{n+1} =e_n$; $n\in \IN$.)
By the tracial condition of $\tau_0$, for any $a\in A_+$,
the sequence $(\tau_1(e_n a e_n))_{n\in \IN}$ is increasing.
To see this, we first observe that for $a_1\in (A_1)_+$ and $n\in \IN$,
\[\tau_0(e_{n} a^2_1 e_{n}) =\tau_0(a_1 e_n^2 a_1)\leq \tau_0(a_1 e_{n+1}^2 a _1) =\tau_0(e_{n+1}a^2_1 e_{n+1}).\]
As $\tau_1$ is bounded on $\mathrm{Her}(e_{n+1}; A)$, we conclude
\[\tau_1(e_n a e_n) \leq \tau_1(e_{n+1} a e_{n+1})\quad{\rm~for~all~}a\in A_+.\]

We now define
\[\tau\colon A_+ \rightarrow [0, \infty]\]
to be
\[\tau(a):= \lim_{n \rightarrow \infty} \tau_1(e_n a e_n)\quad {\rm ~for~}a\in A_+.\]
This gives the desired tracial weight.
Note that as $e_ne_{n+1}=e_n$, $\tau$ is an extension of $\tau_1$ and thus gives an extension of $\tau_0$.
Then it is clear from the definition that $\tau$ is indeed a proper weight on $A$.
The tracial condition of $\tau$ follows from the following calculations:
\begin{align*}
\tau(a^\ast a) &= \lim_{n \rightarrow \infty} \tau_1(e_n a^\ast a e_n)\\
&=\lim_{n \rightarrow \infty} \lim_{m \rightarrow \infty} \tau_1(e_n a^\ast e_m^2 a e_n)\\
&=\lim_{n \rightarrow \infty} \lim_{m \rightarrow \infty} \tau_1(e_m a e_n^2 a^\ast e_m)\\
&\leq \tau(a a^\ast)\quad\quad {\rm~for~all~}a\in A.
\end{align*}

The uniqueness of the extension $\tau$ is clear from the proof of the existence.
\end{proof}

The following connection between KMS weights and tracial weights is crucial for the main purpose of the present article.
\begin{Thm}[\cite{Thom}, Theorem 8.2.13; cf.~ \cite{KK}, Theorem 3.2 and \cite{Tak}, Lemma 8.2]\label{Thm:KK}
Let $\alpha \colon \IR \acts A$ be a flow on a \Cs-algebra $A$.
Let $\beta\in \IR$.
Then there is a bijective correspondence between $\alpha$-$\beta$-KMS weights on $A$
and tracial weights $\tau$ on $A \rtimes_\alpha \IR$ satisfying $\tau \circ \hat{\alpha}_t=e^{-\beta t} \cdot \tau$ for all $t\in \IR$.
Here $\hat{\alpha}\colon \IR \acts A \rtimes_\alpha \IR$ denotes the dual flow of $\alpha$.
\end{Thm}
\begin{Rem}\label{Rem:spl}
By Theorem \ref{Thm:KK}, when there is a faithful $\alpha$-$\beta$-KMS weight for some nonzero $\beta$,
the reduced crossed product $A\rtimes_\alpha \IR$ is stably projectionless,
as observed by Connes \cite{Con}.
Here we recall that a \Cs-algebra $A$ is said to be stably projectionless,
if the stabilization $A \otimes \IK$ does not contain a nonzero projection.
\end{Rem}

\subsection{Compact group actions and fixed point algebras}\label{subsection:compact}
Let $K$ be a compact group.
Then for any (point-norm continuous) action $\alpha \colon K \acts A$ on a \Cs-algebra $A$,
we have the canonical conditional expectation
\[E^\alpha \colon A \rightarrow A^\alpha\]
given by
\[E^\alpha(a):=\int_K \alpha_g(a) dm_K(g)\quad {\rm~for~}a\in A.\]
Here $A^\alpha$ denotes the fixed point algebra of $\alpha$ (we use the same notation for non-compact group actions as well),
and $m_K$ denotes the Haar probability measure on $K$.
For basic facts and some applications of this conditional expectation, see Section 4.5 of \cite{BO} for instance.

\subsection{Almost periodic flows and their KMS weights}
A flow $\alpha \colon \IR \curvearrowright A$ on a \Cs-algebra $A$ is said to be
\emph{almost periodic} if the image $\alpha(\IR) \subset {\rm Aut}(A)$ is relatively compact in the point-norm topology.
We denote by $K_\alpha$ the closure of $\alpha(\IR)$ in ${\rm Aut}(A)$.
For an almost periodic flow $\alpha \colon \IR \curvearrowright A$,
we denote by
\[E^\alpha \colon A \rightarrow A^\alpha\]
the canonical conditional expectation (see Section \ref{subsection:compact})
of the compact group action $K_\alpha \acts A$, which has the same fixed point algebra as $\alpha$.

For almost periodic flows, we have a simple description of KMS weights.
\begin{Prop}\label{Prop:KMS}
Let $\alpha \colon \IR \curvearrowright A$ be an almost periodic flow on a separable \Cs-algebra $A$.
Let $\varphi \colon A_+ \rightarrow [0, \infty]$ be an $\alpha$-KMS weight on $A$.
Then there is a tracial weight $\tau$ on $A^\alpha$ satisfying
\[\varphi=\tau \circ E^\alpha.\]
\end{Prop}
\begin{proof}
Since $\varphi$ is $\alpha$-invariant, by the lower semi-continuity of $\varphi$,
it is $K_\alpha$-invariant.
Then, again by the lower semi-continuity of $\varphi$,
we have
\[\varphi \circ E^\alpha \leq \varphi.\]
This in particular shows the properness of
$\tau:=\varphi |_{A^\alpha_+}$.
By the $\alpha$-KMS condition of $\varphi$,
the weight $\tau$ is tracial.

Choose an approximate unit $(e_n)_{n\in \IN} \subset \Ped(A^\alpha)$ satisfying $e_n e_{n+1}=e_n$ for each $n\in \IN$.
As the inclusion $A^\alpha \subset A$ is non-degenerate,
$(e_n)_{n\in \IN}$ also forms an approximate unit in $A$.
By the positivity of $\varphi$, $\varphi|_{{\rm Her}(e_n; A)}$ is bounded for all $n\in \IN$.
Hence $\varphi(a)=\tau(E^\alpha(a))$ for $a\in \bigcup_{n\in \IN} {\rm Her}(e_n; A)_+$.
For general $a\in A_+$, by the lower semi-continuity of $\varphi$, one has

\begin{align*}
\varphi(a)&\leq \liminf_{n\rightarrow \infty} \varphi(e_n a e_n)\\
 &= \liminf_{n\rightarrow \infty} \tau(E^\alpha(e_n a e_n))\\
&= \liminf_{n \rightarrow \infty} \tau( e_n E^\alpha(a) e_n)\\
&=\liminf_{n \rightarrow \infty} \tau(E^\alpha(a)^{1/2} e_n^2 E^\alpha(a)^{1/2})\\
&= \tau(E^\alpha(a))\quad\quad {\rm~for~all~}a\in A_+.
\end{align*}
As the converse inequality has already been shown,
we conclude $\varphi = \tau \circ E^\alpha$.
\end{proof}

By combining Propositions \ref{Prop:trace} and \ref{Prop:KMS},
we obtain the following simple description of KMS weights for almost periodic flows.
This will be used in our classification results of KMS weights (Theorems \ref{Thm:KMSc} and \ref{Thm:KMSc2}).
Before stating and proving this description, let us introduce one more definition.

Let $\alpha \colon \IR \acts A$ be a flow on a \Cs-algebra $A$.
Let $A_0 \subset \mathcal{A}_\alpha (\subset A)$ be an $\alpha$-invariant $\ast$-subalgebra which also satisfies $\alpha_{-i\beta/2}(A_0)= A_0$.
We say an $\alpha$-invariant linear functional $\varphi \colon A_0 \rightarrow \IC$ satisfies the \emph{$\alpha$-$\beta$-KMS condition on $A_0$}
if it is also $\alpha_{-i\beta/2}$-invariant and satisfies one of the following two equivalent conditions.
\begin{enumerate}
\item For any $a\in A_0$,
$\varphi(a^\ast a)=\varphi(\alpha_{-i\beta/2}(a) \alpha_{-i\beta/2}(a)^\ast)$.
\item For any $a, b\in A_0$, $\varphi(ab)=\varphi(b\alpha_{i\beta}(a))$.
\end{enumerate}
The implication (1) $\Rightarrow$ (2) follows by the same polarization argument as in the proof of (2) $\Rightarrow$ (3) in Theorem 3.2.1 in \cite{Thom},
while the converse implication is clear from the $\alpha_{-i\beta/2}$-invariance of $\varphi$.

\begin{Thm}\label{Thm:KMS}
Let $\alpha \colon \IR \curvearrowright A$ be an almost periodic flow on a \Cs-algebra $A$.
Let $\beta\in \IR$.
Let $A_0 \subset \mathcal{A}_\alpha$ $(\subset A)$ be an $\alpha$-invariant dense $\ast$-subalgebra satisfying the following conditions.
\begin{itemize}\item $\alpha_{-i\beta/2}(A_0)=A_0$.
\item $E^\alpha(A_0)= A_0\cap \Ped(A^\alpha)$.
\item There is a sequence $(e_n)_{n\in \IN} \subset {\rm Ped}(A^\alpha)\cap A_0$
which forms an approximate unit of $A$ and satisfies $e_n e_{n+1}=e_n$ for each $n\in \IN$.
\item $R_n(A_0)\subset A_0$ for all $n\in \IN$.
\end{itemize}
We fix a sequence $(e_n)_{n\in \IN}$ as in the third condition,
and define $A_1:=\bigcup_{n\in \IN} e_n A_0 e_n$.
Then there is a bijective correspondence between $\alpha$-$\beta$-KMS weights $\varphi$ on $A$
and tracial positive linear functionals
$\tau_0$ on $A_1 \cap \Ped(A^\alpha)$
whose composite with $E^\alpha|_{A_1}$, $\tau_0 \circ E^\alpha|_{A_1}$, satisfies the $\alpha$-$\beta$-KMS condition on $A_0$.
The bijection is given by 
\[\varphi \mapsto \tau_0:=\varphi|_{A_1\cap \Ped(A^\alpha)}.\]
\end{Thm}
\begin{proof}
By Proposition \ref{Prop:trace}, any tracial positive linear functional $\tau_0$ on $A_1 \cap \Ped(A^\alpha)$
uniquely extends to a proper tracial weight $\tau$ on $A^\alpha$.
By Proposition \ref{Prop:KMS}, any $\alpha$-KMS weight on $A$ is of the form
$\tau\circ E^\alpha$ for some tracial weight $\tau$ on $A^\alpha$.
Moreover any tracial weight $\tau$ on $A^\alpha$ satisfies ${\rm Ped}(A^\alpha) \subset \mathcal{M}_\tau$.
Therefore $A_1 \subset \cM_\varphi^\alpha$.
Then, by Lemma 3.1.15 of \cite{Thom}, $\varphi|_{A_1}$ is $\alpha_{-i\beta/2}$-invariant.
Hence we only need to prove the next claim:
for a tracial weight $\tau$ on $A^\alpha$,
if the weight $\varphi:=\tau \circ E^\alpha$ satisfies the $\alpha$-$\beta$-KMS condition on $A_1$,
then in fact $\varphi$ satisfies the condition on $A$.

Observe that $\varphi$ is bounded on each ${\rm Her}(e_n; A)$, $n\in \IN$.
(See the first paragraph of the proof of Proposition \ref{Prop:trace}.)
Hence for any $a\in {\rm Her}(e_n; A)$ and $m\in \IN$,
one has
\[\varphi(R_m(a)^\ast R_m(a))= \varphi(\alpha_{-i\beta/2}(R_m(a)) \alpha_{-i\beta/2}(R_m(a))^\ast)\]
because
\[R_m({\rm Her}(e_n; A)) \subset {\rm Her}(e_n; A) \quad {\rm~ for~ all~ }m\in \IN,\quad \alpha_{-i\beta/2}({\rm Her}(e_n; A))\subset {\rm Her}(e_n; A),\]
and the equality is valid in the case $a\in e_n A_0 e_n$.
Then, for any $a\in {\rm Her}(e_n; A) \cap \mathcal{A}_\alpha$, we obtain
\begin{align*}\varphi(a^\ast a)&=\lim_{m \rightarrow \infty} \varphi(R_m(a)^\ast R_m(a))\\
&=\lim_{m \rightarrow \infty} \varphi(\alpha_{-i\beta/2}(R_m(a))\alpha_{-i\beta/2}(R_m(a))^\ast) \\
&=\lim_{m \rightarrow \infty} \varphi(R_m(\alpha_{-i\beta/2}(a)) R_m(\alpha_{-i\beta/2}(a))^\ast)\\
&=\varphi(\alpha_{-i\beta/2}(a)\alpha_{-i\beta/2}(a)^\ast).
\end{align*}
For a general $a\in \cM_{\varphi}^\alpha$,
the lower semi-continuity of $\varphi$ implies
\begin{align*}
\varphi(a^\ast a) &= \liminf_{n \rightarrow \infty} \lim_{m \rightarrow \infty} \varphi(e_n a^\ast e_m^2 a e_n)\\
&= \liminf_{n \rightarrow \infty} \lim_{m \rightarrow \infty} \varphi(e_m\alpha_{-i\beta/2}(a) e_n^2 \alpha_{-i\beta/2}(a)^\ast e_m)\\
&= \varphi(\alpha_{-i\beta/2}(a) \alpha_{-i\beta/2}(a)^\ast).
\end{align*}
Here the first and last equality follow from the lower semi-continuity of $\varphi$ and the inequality
\[\varphi(e_m b e_m) \leq \varphi(b) \quad{\rm~for~}b\in (\cM_\varphi)_+, m\in \IN,\]
which follows from the calculations
\[\varphi(b-e_m b e_m)=\tau(E^\alpha(b)-e_mE^\alpha(b)e_m) =\tau((1-e_m^2)^{1/2}E^\alpha(b)(1-e_m^2)^{1/2}) \geq 0.\]
Thus $\varphi$ satisfies the $\alpha$-$\beta$-KMS condition on $A$.
\end{proof}
\subsection{Pimsner algebras}
For basic facts on Pimsner algebras, we refer the reader to Section 4.6 of \cite{BO} and the original article \cite{Pim}.

Associated to each non-degenerate \Cs-correspondence $\mathcal{E}$ over $A$,
Pimsner \cite{Pim} introduced the universal \Cs-algebra $\mathcal{T}_\mathcal{E}$, called the \emph{Toeplitz--Pimsner algebra},
which is generated by the symbols $T_\xi$; $\xi\in \mathcal{E}$, and a copy of $A$, satisfying the relations
\begin{itemize}
\item $T_{a\xi}=a T_\xi$, $T_{\xi a}= T_\xi a$, $T_{\xi + \eta}=T_\xi + T_\eta$ for $\xi, \eta\in \mathcal{E}$, $a\in A$,
\item $T_\xi^\ast T_\eta= {\ip{\xi, \eta} }$ for $\xi, \eta\in \mathcal{E}$.
\end{itemize}
Note that by these conditions, the inclusion $A \subset \mathcal{T}_\mathcal{E}$ is non-degenerate.

Put $\mathcal{E}^{\otimes 0}:=A$ and we regard it as the \Cs-correspondence over $A$ by the left multiplication action.
For $n\geq 1$, denote by $\mathcal{E}^{\otimes n}$
the $n$-fold tensor product power of $\mathcal{E}$ (as the \Cs-correspondence over $A$).

A concrete realization of $\mathcal{T}_\mathcal{E}$ can be given on the full Fock space
\[\mathrm{F}(\mathcal{E}):=\bigoplus_{n \in \IN} \mathcal{E}^{\otimes n}\]
over $\mathcal{E}$.
The representation $\mathcal{T}_\mathcal{E} \subset \IB_A(\mathrm{F}(\mathcal{E}))$
is given by corresponding $A\subset \mathcal{T}_\mathcal{E}$ to the obvious left $A$-action on $\mathrm{F}(\mathcal{E})$ and sending the symbol $T_\xi$; $\xi\in \mathcal{E}$,
to the (left) creation operator (denoted by the same symbol $T_\xi$)
\[\mathrm{F}(\mathcal{E})\ni \eta \mapsto \xi \otimes \eta.\] 
It is not hard to check that, or it is clear from the realization $\mathcal{T}_\mathcal{E} \subset \IB_A({\mathrm F}(\mathcal{E}))$, that for each $n\in \IN$, the map
\[\xi_1 \otimes \xi_2 \otimes \cdots \otimes \xi_n \mapsto T_{\xi_1} T_{\xi_2} \cdots T_{\xi_n}; \xi_1, \xi_2, \ldots, \xi_n \in \mathcal{E}\]
extends to an isometric linear map 
\[T\colon \mathcal{E}^{\otimes n} \rightarrow \mathcal{T}_\mathcal{E};~ \xi \mapsto T_\xi.\]
The map satisfies $T_\xi T_\eta^\ast = \ip{\xi, \eta}$ for $\xi, \eta \in \mathcal{E}^{\otimes n}$.
Let $P\colon \mathrm{F}(\mathcal{E}) \rightarrow \mathcal{E}^{\otimes 0}$ be the orthogonal projection.
Then it is not hard to check that
$P \mathcal{T}_\mathcal{E} P =\IK_A(\mathcal{E}^{\otimes 0})$.
By identifying $\IK_A(\mathcal{E}^{\otimes 0})$ with $A$ in the obvious way,
we obtain the conditional expectation $E_A \colon \mathcal{T}_\mathcal{E} \rightarrow A$
given by $E_A(a):=PaP; a\in \mathcal{T}_\mathcal{E}$.

The Toeplitz--Pimsner algebra $\mathcal{T}_\mathcal{E}$ has a natural quotient $\mathcal{O}_\mathcal{E}$, called the \emph{Cuntz--Pimsner algebra},
which is given by the ideal generated by $\mathcal{T}_\mathcal{E} \cap \IK_A(\mathcal{E}^{\otimes 0})$.
Here we regard $\IK_A(\mathcal{E}^{\otimes 0})\subset \IK_A(\mathrm{F}(\mathcal{E}))$ and $\mathcal{T}_\mathcal{E} \subset \IB_A(\mathrm{F}(\mathcal{E}))$ in the canonical way.
For $\xi\in \mathcal{E}^{\otimes n};~n\in \IN$, denote by $S_\xi$ the quotient image of $T_\xi$ in $\mathcal{O}_\mathcal{E}$.

By the universality of $\mathcal{T}_\mathcal{E}$ and $\mathcal{O}_\mathcal{E}$, we have natural circle group actions
\[\sigma\colon \IT \acts \mathcal{T}_\mathcal{E}\quad {\rm~and~}\quad \bar{\sigma}\colon \IT \acts \mathcal{O}_\mathcal{E}\]
given by
\[\sigma_z(T_\xi):=zT_\xi \quad {\rm~and~}\quad \bar{\sigma}_z(S_\xi):=zS_\xi \quad {\rm~~for~} z\in \IT,~\xi\in \mathcal{E}.\]
These actions are referred to as the \emph{gauge action} on $\mathcal{T}_\mathcal{E}$, $\mathcal{O}_\mathcal{E}$, respectively.

More generally, any strongly continuous unitary representation $u \colon G\rightarrow \mathcal{U}_A(\mathcal{E}) \cap A'$
of a locally compact group $G$ induces the actions
\[\gamma \colon G \acts \mathcal{T}_\mathcal{E} \quad{\rm~and~}\quad \bar{\gamma} \colon G \acts \mathcal{O}_\mathcal{E}\]
given by
\[\gamma_g(T_\xi):=T_{u_g(\xi)},\quad \bar{\gamma}_g(S_\xi):=S_{u_g(\xi)}\quad {\rm~for~}g\in G, \xi\in \mathcal{E}.\]
These actions are referred to as the \emph{quasi-free actions induced from $u$}.

The fixed point algebras $\mathcal{T}_\mathcal{E}^\sigma$ and $\mathcal{O}_\mathcal{E}^{\bar{\sigma}}$
of the gauge actions play substantial roles in the analysis of the Pimsner algebras $\mathcal{T}_\mathcal{E}$, $\mathcal{O}_\mathcal{E}$.
This is because they have the following tight relation with the coefficient \Cs-algebra $A$.
For $n\in \IN$, define $B_n \subset \mathcal{T}_\mathcal{E}$
to be the \Cs-subalgebra of $\mathcal{T}_\mathcal{E}$
generated by elements $T_\xi T_\eta^\ast$, $\xi, \eta \in \mathcal{E}^{\otimes n}$.
Note that $B_0=A$.
One has a $\ast$-isomorphism $\Phi_n \colon B_n \rightarrow \IK_A(\mathcal{E}^{\otimes n})$
given by
\[T_\xi T_\eta^\ast \mapsto \theta_{\xi, \eta} \quad {\rm ~for~}\xi, \eta \in \mathcal{E}^{\otimes n}.\]
Here $\theta_{\xi, \eta}\in \IK_A(\mathcal{E}^{\otimes n})$
denotes the rank one operator given by
\[\theta_{\xi, \eta}(\zeta):=\xi\ip{\eta, \zeta};\quad \zeta \in \mathcal{E}^{\otimes n}.\]
Hence $B_n$ is strongly Morita equivalent to $A$.
Note that $\Phi_n$ coincides with the canonical action of $B_n$ on $\mathcal{E}^{\otimes n} \subset \mathrm{F}(\mathcal{E})$.

For $n\in \IN$, set
\[B_{\leq n}:=B_0 + B_1 + \cdots + B_n.\]
Then it is known that $B_{\leq n}$ is a \Cs-subalgebra of $\mathcal{T}^\sigma_\mathcal{E}$,
which contains $B_n$ as an ideal.
It is easy to see (by using $E^\sigma$) that the increasing union $\bigcup_{n\in \IN} B_{\leq n}$
is dense in $\mathcal{T}^\sigma_\mathcal{E}$.
Moreover, for each $n\in\IN$, we have the natural (splitting) short exact sequence
\[\{0\} \rightarrow B_n \rightarrow B_{\leq n} \rightarrow B_{\leq n-1} \rightarrow \{0\}.\]
In particular, the sequence $(B_n)_{n\in \IN}$ of subspaces are linearly independent in $\mathcal{T}_\mathcal{E}^\sigma$.

For each $n\in \IN$,
the quotient map $\mathcal{T}_\mathcal{E} \rightarrow \mathcal{O}_\mathcal{E}$
is injective on $B_n$. For short, we denote the (isomorphic) image of $B_n$ in $\mathcal{O}_\mathcal{E}$
by the same symbol.
We remark that, contrary to the Toeplitz--Pimsner algebra case, in $\mathcal{O}_\mathcal{E}$,
the subspaces $(B_n)_{n\in \IN}$ are not necessarily linearly independent.
In Section \ref{section:fr}, we only consider the case
that the inclusions $B_n \subset B_{n+1}$; $n\in \IN$, hold true in $\mathcal{O}_\mathcal{E}$.

\section{A \Cs-correspondence $\E$ associated to a sequence $\fr$ of $\ast$-endomorphisms}\label{section:corr}
In this section, we introduce the ingredients of our construction.
Because it is easy to extend the construction and results to the equivariant setting,
in this and the next three sections, we concentrate on the non-equivariant case.
We discuss the equivariant version of our construction and its basic properties in Section \ref{section:equiv}.

Let $A$ be a separable \Cs-algebra.
Let $\fr=(\rho_n)_{n\in \IN}$ be a sequence of non-degenerate $\ast$-endomorphisms on $A$ satisfying the following two conditions:
\begin{description}
\item[($\diamondsuit1$)] For any $a\in A$, $\limsup_{n \rightarrow \infty}\|\rho_n(a)\|=\|a\|$.
\item[($\diamondsuit2$)] There is no proper ideal $I\subset A$ satisfying
$\rho_n(I)\subset I$ for all $n\in \IN$.
\end{description}
On the right Hilbert \Cs-$A$-module ${\mathcal{E}_{\fr}}:=\ell^2(\IN) \otimes A$,
we define the left $A$-action
\[{\pi_\fr} \colon A \rightarrow \IB_A({\mathcal{E}_{\fr}})\]
to be
\[{\pi_\fr}(a)(\delta_n \otimes b) := \delta_n \otimes \rho_n(a)b \quad{\rm~for~}a, b\in A, n\in \IN.\]
(In other words, $\pi_\fr=\bigoplus_{n \in \IN}\rho_n$.)
With this left $A$-action ${\pi_\fr}$, we regard ${\mathcal{E}_{\fr}}$ as a \Cs-correspondence over $A$.
By the condition ($\diamondsuit1$) of $\fr$, we have
\[{\pi_\fr}(A)\cap \IK_A(\E)=\{0\}.\]
Hence the two Pimsner algebras $\T$ and $\mathcal{O}_{\mathcal{E}_{\fr}}$ are the same.
We use the former notation $\T$,
because some results in the next two sections equally work for the Toeplitz--Pimsner algebras even for the finite rank analogue studied in Section \ref{section:fr}.

For a sequence $\fx=( x_n)_{n\in \IN} \in \IR^\IN$,
we denote by
\[u_\fx \colon \mathbb{R} \rightarrow \mathcal{U}_A({\mathcal{E}_{\fr}})\]
 the strongly continuous unitary representation
given by
\[u_{\fx, t} (\delta_n \otimes a):=e^{i x_n t} \delta_n \otimes a\quad {\rm~for~}a\in A, n\in \IN, t\in \IR.\]
Obviously $u_{\fx}$ commutes with ${\pi_\fr}(A)$.
Hence $u_\fx$ defines the quasi-free flow
\[\gamma_\fx \colon \mathbb{R} \curvearrowright \mathcal{T}_{\mathcal{E}_{\fr}},\]
which is given by
\[\gamma_{\fx, t}(T_{\xi})=T_{u_{\fx, t}(\xi)}\quad {\rm~for~} \xi \in {\mathcal{E}_{\fr}}, t\in \IR.\]
Clearly, the flow $\gamma_{\fx}$ is almost periodic.

We next introduce some notations which are useful in our analysis.
For notational convenience, we put $\IN^0:=\{\emptyset\}$.
We define
\[\fN:=\bigsqcup_{n\in \IN} \IN^n.\]
For each $\emptyset \neq \n=(i_1, i_2, \ldots, i_n) \in \fN$,
we define
\[|\n|:=n,\]
\[\rho_\n:=\rho_{i_n} \circ \rho_{i_{n-1}} \circ \cdots \circ \rho_{i_1},\]
\[ x_\n:=\sum_{k=1}^n x_{i_k}.\]
For notational convenience, we set
\[|\emptyset|:=0,\quad \rho_\emptyset:= \id_A, \quad x_\emptyset:=0.\]
For $n\in \IN$,
put
\[\fN_{\leq n}:=\{\n\in \fN: |\n|\leq n\}.\]
For $\n=(i_1, \ldots, i_n), \m=(j_1, \ldots, j_m) \in \fN$,
put $\n \star \m :=(i_1, \ldots, i_n, j_1, \ldots, j_m)\in \fN$. 

Observe that for any $n\in \IN$, we have an isomorphism
\[\mathcal{E}_{\fr}^{\otimes n} \cong \ell^2(\IN^ n) \otimes A\]
of right Hilbert \Cs-$A$-modules given by
\[(\delta_{i_1}\otimes a_1) \otimes (\delta_{i_2}\otimes a_2) \otimes \cdots \otimes (\delta_{i_n}\otimes a_n) \mapsto \delta_{\n} \otimes \rho_{\n_1}(a_1)\rho_{\n_2}(a_2)\cdots \rho_{\n_{n}}(a_n)\]
for $\n=(i_1, i_2, \ldots, i_n)\in \IN^n$ and $a_1, \ldots, a_n \in A$, where $\n_k:=(i_{k+1}, i_{k+2}, \cdots, i_n)$ for $k=1, \ldots, n$.
Throughout this article, we frequently identify $\mathcal{E}_{\fr}^{\otimes n}$ with $\ell^2(\IN^ n) \otimes A$ via this isomorphism without mentioned.
Note that in the latter picture, the left $A$-action $\pi_\fr$ and the unitary representation $u_{\fx}^{\otimes n}$ are given by
\[{\pi_\fr}(a)(\delta_\n \otimes b)=\delta_\n \otimes \rho_\n(a)b,\quad u_{\fx, t}^{\otimes n}(\delta_\n \otimes b)=e^{ix_\n t }\delta_\n \otimes b\]
for $a, b\in A$, $\n \in \IN^n$, $t\in \IR$.

Note that for any $\n\in \fN$ and $a\in A$,
the element $T_{\delta_\n\otimes a}$ is $\gamma_\fx$-entire.
Indeed one has the analytic extension
\[\gamma_{\fx, z}(T_{\delta_\n \otimes a}) :=e^{i x_\n z} T_{\delta_\n\otimes a},\quad z\in \IC\]
of the orbit map $t \mapsto \gamma_{\fx, t}(T_{\delta_\n \otimes a})$.

Take an approximate unit $(e_n)_{n\in \IN}$ of $A$.
For each $\n \in \fN$, 
define an isometric element $T_\n\in \cM(\T)$
to be the strict limit of the sequence
$(T_{\delta_\n\otimes e_n})_{n \in \IN}$ in $\T$.
(It is not hard to check that this sequence in fact converges to an isometric element in the strict topology,
and the limit does not depend on the choice of $(e_n)_{n\in \IN}$.)
Alternatively, in the concrete realization $\T \subset \IB_A({\rm F}(\E))$,
the element $T_{\n} \in \cM(\T) \subset \IB_A({\rm F}(\E))$ corresponds to the creation operator
\[\xi \mapsto \delta_\n \otimes \xi;~\xi\in {\rm F}(\E).\]
Note that for any $a\in A$ and any $\n\in \fN$,
a direct calculation shows
\[a T_\n =T_\n \rho_\n(a).\]

\section{Simplicity of $\T$ and $\T\rtimes_{\gamma_\fx} \IR$}
In this section, we prove the simplicity of $\T$.
Then, under mild conditions on $\fr$ and $\fx$, we further show
the simplicity of the reduced crossed product \Cs-algebra $\T\rtimes_{\gamma_\fx} \IR$.

In the proof of these simplicity results, the gauge fixed point algebra $\T^{\sigma}$ plays a crucial role.
We first study its basic properties.
\begin{Lem}\label{Lem:nondeg}
For any $n \in \IN$, the canonical representation $\pi_n \colon B_{\leq n} \rightarrow \IB_A(\E^{\otimes n})$ is faithful.
In particular, the ideal $B_n \subset B_{\leq n}$ is essential.
\end{Lem}
\begin{proof}The case $n=0$ is trivial, hence we may assume that $n \geq 1$.
Take $b\in \ker(\pi_n)$.
Write $b=\sum_{i=0}^n b_i$; $b_i\in B_i$ for $i=0, 1, \ldots, n$.
For each $1\leq k\leq n$, let $\mathfrak{F}(\IN^k)$ denote the directed set of all finite subsets of $\IN^k$.
For $F\in \mathfrak{F}(\IN^k)$, set $P_{F}:=\sum_{\n \in F} T_\n T_\n^\ast$, and
let $F^{(k)} \subset \IN$ denote the image of $F$ by the $k$th coordinate projection $\IN^k \rightarrow \IN$.
Note that for any $F\in \mathfrak{F}(\IN^k)$, one has
$ \E^{\otimes k-1}\otimes P_{F^{(k)}} \E ^{\otimes n-k+1} \subset P_F(\E^{\otimes n}).$

On the one hand, for any $k \leq i \leq n$, one has
\[\lim_{F\in \mathfrak{F}(\IN^k)} b_i P_F =b_i.\]
On the other hand, for any $F\in\mathfrak{F}(\IN^k)$, we have
\[\| \pi_n( \sum_{i=0}^{k-1} b_i(1-P_F))\|=\|\pi_n( \sum_{i=0}^{k-1} b_i)\|.\]
Indeed, by condition $(\diamondsuit1)$ of $\fr$, the $\ast$-homomorphisms
\[\IB_A(\E^{\otimes {k-1}}) \rightarrow \IB_A( \E^{\otimes n}) \quad {\rm~and~} \quad \IB_A(\E^{\otimes {k-1}}) \rightarrow \IB_A(\E^{\otimes k-1}\otimes P_{F^{(k)}} \E ^{\otimes n-k+1})\]
given by
\[a\mapsto a \otimes 1_{\E^{\otimes n-k+1}} \quad{\rm~and~}\quad a\mapsto a \otimes 1_{P_{F^{(k)}} \E ^{\otimes n-k+1}}\]
are injective, hence both norms are equal to $\|\pi_{k-1}(\sum_{i=0}^{k-1} b_i)\|$.

These observations force $b_i\in \ker(\pi_n)$ for all $i=0, 1, \ldots, n$.
Again by condition $(\diamondsuit1)$ of $\fr$,
$\pi_n$ is isometric on each $B_i$; $i=0, 1, \ldots, n$.
Hence $b=0$.
\end{proof}
\begin{Lem}\label{Lem:giid}
There is no proper ideal $I$ of $\T^\sigma$ satisfying
$T_\n ^\ast\cdot I \cdot T_\n \subset I$ for all $\n\in \fN$.
\end{Lem}
\begin{proof}
Let $I \subset \T^\sigma$ be an ideal as in the statement.
Since the increasing union $\bigcup_{n\in \IN} B_{\leq n}$ is dense in $\T^\sigma$,
one can take $n\in \IN$ satisfying
\[B_{\leq n} \cap I \neq \{0\}.\]
By Lemma \ref{Lem:nondeg}, this implies $B_n \cap I \neq \{0\}$. 
Then, for any $\n \in \IN^n$, we have
\[\{0\} \neq T_\n ^\ast \cdot (B_n \cap I) \cdot T_\n \subset A \cap I.\]
Moreover, for any $n\in \IN$, we have
\[\rho_n(A \cap I)=T_n ^\ast \cdot (A\cap I) \cdot T_{n} \subset A\cap I.\]
This together with condition ($\diamondsuit2$) of $\fr$ implies
$A \subset I$.
Hence
$I=\T^\sigma$.
\end{proof}
\begin{Thm}\label{Thm:Tsimple}
The Toeplitz--Pimsner algebra $\mathcal{T}_{\mathcal{E}_{\fr}}$ is simple.
\end{Thm}
\begin{proof}
For $n\in \IZ$, set
\[V_n:=\spa\left\{T_\xi T_\mu^\ast : \xi\in \mathcal{E}_{\fr}^{\otimes k}, \mu \in \mathcal{E}_{\fr}^{\otimes l}:k, l \in \IN, k-l = n\right\} \subset \T.\]
Note that $A \subset V_0$.
For $n\in \IZ$ and $m \geq |n|$, set
\[V_{n, \leq m}:=\spa\left\{T_\xi T_\mu^\ast : \xi\in \mathcal{E}_{\fr}^{\otimes k}, \mu \in \mathcal{E}_{\fr}^{\otimes l}:0 \leq k, l \leq m, k-l = n\right\} \subset V_n.\]
Clearly $V_{n, \leq m} \subset V_{n, \leq m+1}$ for $m\geq |n|$,
and $\bigcup_{m= |n|}^\infty V_{n, \leq m}=V_n$ for each $n \in \IN$.
It is easy to see that
$\spa\{V_n : n \in \IZ\}$ is a dense $\ast$-subalgebra of $\T$.

Now let $I \subset \T$ be a nonzero ideal of $\T$.
Observe that $E^\sigma(I)$ is a nonzero algebraic ideal of $\T^\sigma$.
By the definition of $E^\sigma$, one has
\[T_\n^\ast\cdot E^\sigma(I) \cdot T_\n=E^\sigma(T_\n^\ast\cdot I\cdot T_\n) \subset E^\sigma(I)\] for all $\n \in \fN$.
Thus, by Lemma \ref{Lem:giid}, $E^\sigma(I)$ is dense in $\T^\sigma$.
Since $\Ped(A) \subset \Ped(\T^\sigma)$,
this implies $A \cap E^\sigma(I) \neq \{0\}$.

Take $x\in I$ satisfying
\[E^\sigma(x)\in A,\quad \|E^\sigma(x)\|=1.\] 
Choose $ x_n \in V_{n, \leq N}$; $N\in \IN$, $n=-N, -N+1, \ldots, N$,
satisfying
\[x \approx_{1/2} \sum_{n=-N}^N x_n, \quad x_0=E^\sigma(x).\]
By condition ($\diamondsuit1$) of $\fr$, one can find $\n=(i_1, \ldots, i_{N+1})\in \IN^{N+1}$
whose entries are pairwise distinct and satisfies $\|\rho_\n(x_0)\|>1/2$.
By the first condition of $\n$, we have
\[T_\n^\ast \cdot V_{n, \leq N} \cdot T_\n = \{0 \} \quad{\rm~for~all~} n\in \IZ {\rm~with~} 1 \leq |n|\leq N.\]
This implies
\[I \ni T_\n^\ast x T_\n \approx_{1/2} \sum_{n=-N}^N T_\n^\ast x_n T_\n = T_\n ^\ast x_0 T_\n =\rho_\n(x_0) \in A \subset \T^\sigma.\]
Since the norm of $\rho_\n(x_0)$ is greater than $1/2$,
this shows that the quotient map $A \rightarrow A/(A\cap I) \subset \T/I$ is not isometric (at $\rho_\n(x_0)$).
Hence $I \cap \T^\sigma \neq \{0\}$.
This together with Lemma \ref{Lem:giid} yields
\[I= \T.\]
This proves the simplicity of $\T$.
\end{proof}
We next study the simplicity of $\T \rtimes_{\gamma_\fx} \IR$.
Denote by $D(\IT^\IN)$ the subgroup of $\IT^\IN$ consisting of all constant sequences.
Let $K_\fx$ be the closure of $\{(e^{i x_n t})_{n\in \IN}: t\in \mathbb{R}\}$
 in the compact group $\IT^\IN$.
\begin{Thm}\label{Thm:Rsimple}
Assume that the pair $(\fr, \fx)$ satisfies the following two conditions.
\begin{enumerate}[\upshape(1)]
\item $D(\IT^\IN) \subset K_\fx$.
\item For any nonzero ideal $I$ of $A$, there is a nonzero ideal $J$ of $A$ satisfying the following condition:
The set $\{x_\n:\n\in \fN, J \subset \cI(\rho_\n(I); A)\}$ is dense in $\IR$.
\end{enumerate}
Then the reduced crossed product $\T \rtimes_{\gamma_\fx} \IR$ is simple.
\end{Thm}
\begin{Exm}\label{Exm:dense}
If the sequence $\fx=(x_n)_{n\in \IN}\subset \IR$ is linearly independent over $\mathbb{Q}$,
then $K_\fx=\IT^\IN$.
If we further assume that $I \subset \cI(\rho_i(I); A)$ for any ideal $I$ of $A$ and $i=0, 1$, and that $x_0 x_1 <0$,
then for any nonzero ideal $I$ of $A$,
the set 
\[\{x_\n :\n\in \fN, I \subset \cI(\rho_\n(I); A)\}\]
 contains the subsemigroup of $\IR$ generated by $x_0$ and $x_1$, which is dense in $\IR$.
Thus the pair $(\fr, \fx)$ satisfies the assumptions of Theorem \ref{Thm:Rsimple}.
\end{Exm}

To prove Theorem \ref{Thm:Rsimple}, we first study ideals
of the reduced crossed products of some subsystems of $\gamma_\fx$.
\begin{Lem}\label{Lem:idT}
For any nonzero ideal $I$ of $\T^\sigma \rtimes_{\gamma_\fx} \IR$
satisfying $T_\n^\ast \cdot I \cdot T_\n \subset I$ for all $\n \in \fN$,
one has $I \cap (A \rtimes_{\gamma_\fx} \IR) \neq \{0\}$.
\end{Lem}
\begin{proof}
We first show that $I$ has a nonzero intersection with $B_n \rtimes_{\gamma_\fx} \IR$ for some $n\in \IN$.
Since $\T^\sigma \rtimes_{\gamma_\fx} \IR$ is the closure of the increasing union
$\bigcup_{n\in \IN} (B_{\leq n} \rtimes_{\gamma_\fx} \IR)$,
one can find $n\in \IN$ satisfying
\[(B_{\leq n} \rtimes_{\gamma_\fx} \IR) \cap I \neq \{0\}.\]
By Lemma 4.1, the canonical representation $B_{\leq n} \subset \IB_A(\E^{\otimes n})$
gives rise to the faithful regular representation $B_{\leq n} \rtimes_{\gamma_\fx} \IR \subset \IB_A(L^2(\IR)\otimes \E^{\otimes n})$.
Since this representation is non-degenerate on $B_n \rtimes_{\gamma_\fx} \IR$,
we obtain
\[I \cap (B_n \rtimes_{\gamma_\fx} \IR) \neq \{0\}.\]

Let $\Phi \colon B_n \rightarrow \IK_A(\E^{\otimes n})$ denote the canonical $\ast$-isomorphism.
Observe that the flow $\gamma_\fx \colon \IR \acts B_n$ is inner:
it is given by the adjoint of the unitary representation $\Phi^{-1}\circ u_\fx^{\otimes n} \colon \IR \rightarrow \cM(B_n)$.
Hence we have a $\ast$-isomorphism
\[\Psi \colon B_n \rtimes_{\gamma_\fx} \IR \rightarrow \IK_A(\E^{\otimes n}) \otimes \Cso(\IR)\]
given by
\[\Psi(f)(t):= \Phi(f(t))u_{\fx, t}^{\otimes n} \quad {\rm~for~}f\in C_c(\IR, B_n), t\in \IR.\]
Here we identify $C_c(\IR, \IK_A(\E^{\otimes n}))$ with a $\ast$-subalgebra of $\IK_A(\E^{\otimes n}) \otimes \Cso(\IR)$ in the obvious way.
Direct calculations show that
\[\Psi(T_\xi a T_\eta^\ast)=\theta_{\xi, \eta} \otimes a\quad {\rm~for~}\xi, \eta \in \E^{\otimes n}, a\in \Cso(\IR).\]

We next choose nonzero ideals $J \subset \IK_A(\E^{\otimes n})$ and $K \subset \Cso(\IR)$ satisfying
\[J \otimes K \subset \Psi(I \cap (B_n \rtimes_{\gamma_\fx} \IR)).\]
(For the existence of such $J$, $K$, see e.g., Corollary 9.4.6 of \cite{BO}.)
By the canonical $\ast$-isomorphism $\IK_A(\E^{\otimes n}) \cong \IK(\ell^2(\IN^n)) \otimes A$,
there is a (nonzero) ideal $J_0$ of $A$ satisfying
\[J=\overline{\rm span}\{\theta_{\xi, \eta}: \xi, \eta \in \E^{\otimes n} J_0\}.\]
This shows that $T_\xi a T_\eta^\ast \in I$ for all $\xi, \eta \in \E^{\otimes n} J_0$ and $a\in K$.
We choose any $\n\in \IN^n$.
Then, by the assumption on $I$, we obtain
\[\{0\}\neq J_0\cdot K \subset T_\n^\ast \cdot (\Psi^{-1}(J\otimes K))\cdot T_\n \subset I \cap (A\rtimes_{\gamma_\fx}\IR).\]
\end{proof}

The second condition of Theorem \ref{Thm:Rsimple} is crucial in the next lemma.
\begin{Lem}\label{Lem:idBn}
Assume that the pair $(\fr, \fx)$ satisfies condition $(2)$ of Theorem \ref{Thm:Rsimple}.
Then for any nonzero ideal $I$ of $A \rtimes_{\gamma_\fx} \IR$ satisfying $T_\n^\ast \cdot I \cdot T_\n\subset I$ for all $\n\in \fN$,
there is a nonzero ideal $J \subset A$ satisfying
\[J \rtimes_{\gamma_\fx} \IR \subset I.\]
\end{Lem}
\begin{proof}
Note that the flow $\gamma_{\fx}$ is trivial on $A$.
We therefore identify $A\rtimes_{\gamma_\fx} \IR$ with $A \otimes \Cso(\IR)$ in the obvious way.

Let $I \subset A \otimes \Cso(\IR)$ be an ideal as in the statement. 
Take nonzero ideals $K \subset A$ and $L \subset \Cso(\IR)$
satisfying $K \otimes L \subset I$.
(For the existence of such $K$, $L$, see e.g., Corollary 9.4.6 of \cite{BO}.)
By applying the present assumption of $(\fr, \fx)$ (condition $(2)$ in Theorem \ref{Thm:Rsimple}) to $K$,
one can choose a nonzero ideal $J$ of $A$ such that the set
\[X:=\left\{x_\n:\n\in \fN, J\subset \cI(\rho_{\n}(K); A)\right\}\]
 is dense in $\IR$.
Observe that for any $\n\in \fN$ with $J\subset \cI(\rho_{\n}(K); A)$,
 direct calculations show that
\[J \otimes \varrho_{x_\n}(L) \subset \cI(T_\n^\ast\cdot (K\otimes L) \cdot T_\n; A\otimes \Cso(\IR)) \subset I.\]
Here
\[\varrho \colon \IR \curvearrowright \Cso(\IR)\]
denotes the flow given by
\[\varrho_s(f)(t):=e^{-ist}f(t),\quad f\in C_c(\IR), s, t\in \IR.\]
It is not hard to check that the Fourier transformation
$\Cso(\IR) \cong C_0(\IR)$ transforms the flow $\varrho$ into the translation flow
\[\mathrm{L}\colon \IR\acts C_0(\IR),\quad (\mathrm{L}_s(f))(t):=f(t-s); f\in C_0(\IR), s, t\in \IR.\]
Since $X$ is dense in $\IR$, this yields
\[\cI(\bigcup_{x\in X} \varrho_x(L); \Cso(\IR))=\Cso(\IR).\]
Consequently the ideal $J$ possesses the desired property.
\end{proof}

\begin{proof}[Proof of Theorem \ref{Thm:Rsimple}]
Let $\tilde{\sigma}\colon\IT \acts \T\rtimes_{\gamma_\fx} \IR$
be the action given by
\[(\tilde{\sigma}_z(f))(t):=\sigma_z(f(t)), \quad f \in C_c(\IR, \T),~ z\in \IT,~ t\in \IR.\]
By condition (1),
for every $z\in \IT$, $\tilde{\sigma}_z$ is approximately inner.
This shows that any ideal $I$ of $\T\rtimes_{\gamma_\fx} \IR$ is $\tilde{\sigma}$-invariant
and hence satisfies $E^{\tilde{\sigma}}(I) \subset I$.
Thus, for any nonzero ideal $I\subset \T\rtimes_{\gamma_\fx} \IR$,
the subset $E^{\tilde{\sigma}}(I)$
is a nonzero ideal of $(\T\rtimes_{\gamma_\fx} \IR)^{\tilde{\sigma}}= E^{\tilde{\sigma}}(\T\rtimes_{\gamma_\fx} \IR) =\T^\sigma \rtimes_{\gamma_\fx} \IR$.
By applying Lemmas \ref{Lem:idT} and \ref{Lem:idBn} to $E^{\tilde{\sigma}}(I)$ $(\subset I)$, one can find a nonzero ideal $J \subset A$
 satisfying
\[J \rtimes_{\gamma_\fx} \IR \subset I.\]
Since $\T$ is simple (Theorem \ref{Thm:Tsimple}),
this implies
\[I= \T\rtimes_{\gamma_\fx} \IR.\]
We conclude the simplicity of $\T\rtimes_{\gamma_\fx} \IR$.
(An alternative direct proof: We have
\[\rho_\n(J)\rtimes_{\gamma_\fx} \IR =T_\n^\ast\cdot (J \rtimes_{\gamma_\fx} \IR ) \cdot T_\n \subset I \quad{\rm~for~all~}\n\in \fN,\]
whence $I=\T \rtimes_{\gamma_\fx} \IR$ by condition $(\diamondsuit 2)$ of $\fr$.)
\end{proof}

\begin{Rem}
The above proof shows that condition (2) of Theorem \ref{Thm:Rsimple} can be replaced by the following slightly weaker condition:\\
$(2')$: For any $a\in A \setminus \{0\}$ and any $b\in \Cso(\IR) \setminus\{0\}$, one has
\[\cI(\{\rho_\n(a)\otimes \varrho_{x_\n}(b): \n \in \fN\}; A \otimes \Cso(\IR)) = A \otimes \Cso(\IR).\]
In other words, the semigroup $\{(\rho_\n)_\ast \times {\rm L}_{x_\n}: \n \in \fN \}$
of continuous maps on the topological space $ {\rm Prim}(A) \times \IR$ is \emph{minimal}.
\end{Rem}

\section{Classification of $\gamma_\fx$-KMS weights and tracial weights on $\T$}
In this section we classify all $\gamma_\fx$-KMS weights on $\T$,
under the condition
\[D(\IT^\IN) \subset K_\fx.\]
By Theorem \ref{Thm:KK}, this in particular gives
a useful necessary condition for the existence of a tracial weight on $\T \rtimes_{\gamma_\fx} \IR$.
By using this classification result, we also determine the trace space ${\rm T}(\T)$ for certain $(A, \fr)$.

\begin{Thm}\label{Thm:KMSc}
Assume that $D(\IT^\IN) \subset K_\fx$.
Then for every $\beta\in \IR$, there is a bijective correspondence
between $\gamma_\fx$-$\beta$-KMS weights $\varphi$ on $\T$ and
tracial weights $\tau$ on $A$ satisfying the inequality
\[(\clubsuit\mathchar`-\beta)\quad \sum_{n\in \IN} e^{-\beta x_n}\tau(\rho_n(a))\leq \tau(a) \quad {\rm~for~all~}a\in \Ped(A)_+.\]
$($The latter set, hence both sets, would be empty.$)$
The bijection is given by
\[\varphi \mapsto \tau:= \varphi|_{A_+}.\]
\end{Thm}

\begin{Exm}\label{Exm:KMST}
For any sequences $\fx=(x_n)_{n\in \IN} \in (\IR\setminus \{0\})^\IN$, $(c_n)_{n\in \IN} \in (0, \infty)^\IN$, and any
$\beta\in \IR \setminus \{0\}$,
there is a sequence $(r_n)_{n \in \IN} \in (\mathbb{Q}\setminus \{0\})^\IN$ satisfying
\[\sum_{n\in \IN} c_n e^{-\beta r_n x_n} \leq 1.\]
If we further assume that
$c_0 + c_1 <1$,
then one can require $r_0$ and $r_1$ to be positive.
Thus, if $A$ admits a tracial weight $\tau$ and positive constants $c_n$; $n\in \IN$
satisfying
\[c_0 + c_1 <1, \quad\quad \tau \circ \rho_n(a)\leq c_n \tau(a) \quad{\rm~ for~ all~}a\in {\rm Ped}(A)_+, n\in \IN,\]
then for any $\beta\in \IR \setminus \{0\}$, there exists $\fx=(x_n)_{n\in \IN}\in \IR^\IN$ which is linearly independent over $\mathbb{Q}$, $x_0 x_1 <0$,
and satisfies the inequality ($\clubsuit$-$\beta$) in Theorem \ref{Thm:KMSc} for $\tau$.
For such $\fx$, under the additional assumption that $I \subset \cI(\rho_i(I); A)$ for all ideals $I\subset A$ and $i=0, 1$,
the reduced crossed product $\T \rtimes_{\gamma_\fx} \IR$
is simple and stably projectionless by Remark \ref{Rem:spl}, 
the observation in Example \ref{Exm:dense}, and Theorems \ref{Thm:KK}, \ref{Thm:Rsimple}, \ref{Thm:KMSc}.

If $\sum_{n=1}^\infty c_n \leq 1$, then $\tau$ satisfies the inequality $(\clubsuit$-$0$) for any $\fx\in \IR^\IN$.
Hence both $\T$ and $\T\rtimes _{\gamma_\fx} \IR$ admit a tracial weight.
\end{Exm}
\begin{Rem}
In the easiest case that $A=\IC$ and $\fr=(\id_{\IC})_{n\in \IN}$,
the flows
$\gamma_\fx \colon \IR \acts \T$
coincide with the flows on $\mathcal{O}_\infty$ studied in \cite{Kat}.
Proposition 7.4 in \cite{Kat} shows that their crossed products can never be stably finite simple.
This suggests the significance of the trace shrink property of some members in $\fr$ to obtain stably finite simple crossed product \Cs-algebras (see Example \ref{Exm:KMST}).
\end{Rem}

\begin{proof}[Proof of Theorem \ref{Thm:KMSc}]
Fix $\beta\in \IR$. We first show the injectivity of the map in the statement.
Let $\varphi$ be a $\gamma_\fx$-$\beta$-KMS weight on $\T$.
By the lower semi-continuity of $\varphi$ and the assumption $D(\IT^\IN) \subset K_\fx$, it is invariant under the gauge action $\sigma$.
Observe that the elements $T_{\delta_\n \otimes a}$; $\n \in \fN$, $a\in {\rm Ped}(A)$, are contained in $\cM_\varphi^{\gamma_\fx}$,
because $|T_{\delta_\n \otimes a}|=|a| \in \cM_\varphi$ by Proposition \ref{Prop:KMS}.
Then, for any $a, b \in \Ped(A)$ and $\n, \m \in \fN$ with $|\n|\neq |\m|$, by the gauge invariance of $\varphi$,
we have
\begin{equation}
\varphi(T_{\delta_\n\otimes a} T_{\delta_\m \otimes b}^\ast)=0.\label{equation:KMS1}
\end{equation}
Since $\varphi|_{\T^{\gamma_\fx}}$ is a proper tracial weight (by Proposition \ref{Prop:KMS}) and $\Ped(A)\subset \Ped(\T^{\gamma_\fx})$,
the restriction
$\tau:=\varphi|_{A_+}$ is also a proper tracial weight on $A$.
For $a, b\in \Ped(A)$ and $\n, \m\in \fN$ with $|\n|=|\m|$,
the KMS condition implies
\begin{equation}
\varphi(T_{\delta_\n\otimes a} T_{\delta_\m\otimes b}^\ast)=e^{-\beta x_\n}\varphi(T_{\delta_\m \otimes b}^\ast T_{\delta_\n \otimes a})=e^{-\beta x_\n}\delta_{\n, \m} \tau(b^\ast a).\label{equation:KMS2}
\end{equation}

Next we define
a dense $\ast$-subalgebra $\cS\subset \T$ to be
\[\cS:=\spa\{T_{\delta_\n\otimes a} T_{\delta_\m \otimes b}^\ast: \n, \m \in \fN, a, b\in \Ped(A)\}.\]
Since
$\Ped(A)=\spa(\Ped(A) \cdot \Ped(A))$, one has
\[{\rm Ped}(A)\subset \cS=\spa\{ T_\n a T_\m^\ast: \n, \m \in \fN, a\in \Ped(A)\}.\]
It is not hard to check that
\[E^{\gamma_\fx}(\cS)=\spa\{T_\n a T_\m^\ast: \n, \m \in \fN, a \in \Ped(A), x_\n= x_\m\}\subset \cS.\]
Note that $\cS$ satisfies the assumptions of Theorem \ref{Thm:KMS} for $\gamma_\fx$ and $\beta$.
Hence by Theorem \ref{Thm:KMS}, there exists at most one $\gamma_\fx$-$\beta$-KMS weight on $\T$ satisfying the equations (\ref{equation:KMS1}) and (\ref{equation:KMS2}).
Hence the map $\varphi \mapsto \varphi|_{A_+}$
gives an injective map from the space of $\gamma_\fx$-$\beta$-KMS weights on $\T$ to the trace space of $A$.

We next show that, for any $\gamma_\fx$-$\beta$-KMS weight $\varphi$,
the tracial weight $\tau:=\varphi|_{A_+}$ satisfies the inequality ($\clubsuit$-$\beta$).
For any $a\in {\rm Ped}(A)_+$ and any $N\in \IN$, by the definition of the left $A$-action ${\pi_\fr}$, one has
\[\sum_{n=1}^N T_n \rho_n(a) T_{n}^\ast \leq a.\]
Then the positivity of $\varphi$ and the equation (\ref{equation:KMS2}) (applied to ($n, n, \rho_n(a)^{1/2}, \rho_n(a)^{1/2}$) instead of $(\n, \m, a, b)$) yield
\[\sum_{n=1}^N e^{-\beta x_n}\tau(\rho_n(a)) \leq \tau(a).\] 
Letting $N \rightarrow \infty$, we obtain the inequality ($\clubsuit$-$\beta$).

Now, to complete the proof, we only need to show that for any tracial weight $\tau$ on $A$
satisfying the inequality ($\clubsuit$-$\beta$),
there is a $\gamma_\fx$-$\beta$-KMS weight $\varphi$
with $\varphi|_{A_+}=\tau$.

To construct the desired weight, we use the following presentation of elements in $\cS$.
Since the subspaces $T_\n\cdot A\cdot T_\m^\ast; \n, \m\in \fN$, are linearly independent in $\T \subset \IB_A(\mathrm{F}(\cE))$,
for any $a\in \cS$, there is a unique indexed family $(a_{\n, \m})_{\n, \m \in \fN}$ of elements in $\Ped(A)$
whose all but finitely many entries are $0$, and satisfies
\[a=\sum_{\n, \m\in \fN} T_\n a_{\n, \m} T_\m^\ast.\]
Now, we define
a self-adjoint linear functional
\[\varphi \colon \cS \rightarrow \IC\]
by the formula
\[\varphi(a):= \sum_{\n \in \fN}e^{-\beta x_\n}\tau(a_{\n, \n}) \quad {\rm~for~}a=\sum_{\n, \m \in \fN} T_\n a_{\n, \m} T_\m^\ast\in \cS.\]
We next show that $\varphi$ satisfies the $\gamma_\fx$-$\beta$-KMS condition on $\cS$.
To see this, for $\n, \m, \fk, \fl\in \fN$ and $a, b\in \Ped(A)$,
let us compute the two values $\varphi((T_\n a T_\m^\ast )(T_\fk b T_\fl^\ast))$ and $\varphi((T_\fk b T_\fl^\ast)(T_\n a T_\m^\ast))$.
Note that by the definition of $\varphi$ and the relations in $\T$, both values vanish
when $|\n|- |\m|\neq |\fk|-|\fl|$.
In particular the KMS condition holds true in this case.
We next consider the case $|\n|- |\m|= |\fk|-|\fl|$.
By switching the pairs $(\n, \m)$ and $(\fk, \fl)$ if necessary,
we may further assume that $|\n|\geq |\fk|$, (and hence) $|\m|\geq |\fl|$.
Then notice that the first value $\varphi((T_\n a T_\m^\ast )(T_\fk b T_\fl^\ast))$ vanishes
unless there exists $\mathfrak{p}\in \fN$
satisfying
\[\n=\fl\star \mathfrak{p},\quad \m=\fk \star \mathfrak{p}.\] 
The same condition is also a necessary condition to get
$\varphi((T_\fk b T_\fl^\ast)(T_\n a T_\m^\ast)) \neq 0$.
Therefore we only need to check the KMS condition
under the additional assumption that such a $\mathfrak{p}$ exists.
In this case, direct calculations show that
\begin{align*}
\varphi((T_\n a T_\m^\ast )(T_\fk b T_\fl^\ast))
&=\varphi(T_\n a T_{\mathfrak{p}}^\ast b T_\fl^\ast)\\
&=\varphi(T_\n a \rho_{\mathfrak{p}}(b)T_\n^\ast)\\
&=e^{-\beta x_\n}\tau(a\rho_{\mathfrak{p}}(b)),
\end{align*}
\begin{align*}
\varphi((T_\fk b T_\fl^\ast)(T_\n a T_\m^\ast))
&=\varphi(T_\fk b T_{\mathfrak{p}} a T_\m^\ast)\\
&=\varphi(T_\m \rho_{\mathfrak{p}}(b)a T_\m^\ast)\\
&=e^{-\beta x_\m}\tau(a\rho_{\mathfrak{p}}(b)).
\end{align*}
Since $\gamma_{\fx, i\beta} (T_\n aT_\m^\ast )=e^{-\beta (x_\n - x_\m)} T_\n a T_\m^\ast$, 
this proves the KMS condition.

Now, for the completion of the proof, by Theorem \ref{Thm:KMS}, it suffices to show the positivity of $\varphi$.
For $n\in \IN$, set $\cB_{\leq n}:=B_{\leq n} \cap \cS$.
To show the positivity, since $\varphi=\varphi\circ E^{\sigma}|_{\cS}$ and $E^\sigma(\cS) =\bigcup_{n\in \IN} \cB_{\leq n}$,
it suffices to show the positivity of $\varphi|_{\cB_{\leq n}}$ for each $n\in \IN$.
We prove this by induction on $n\in \IN$.
In the case $n=0$, the statement is trivial because
$\varphi|_{\cB_0}=\tau|_{\Ped(A)}$.
Assume that we have shown the positivity of $\varphi |_{\cB_{\leq n}}$ for a given $n \in \IN$.
Take any positive element
\[a=\sum_{\n, \m\in \fN_{\leq n+1}} T_\n a_{\n, \m} T_\m^\ast \in (\cB_{\leq n+1})_+.\]
Then, for each $k\in \IN$,
a direct calculation shows
\[T_k^\ast a T_k=\rho_k(a_{\emptyset, \emptyset})+ \sum_{\n, \m\in \fN_{\leq n}} T_{\n} a_{k\star\n, k\star\m} T_{\m}^\ast\in (\cB_{\leq n})_+.\]
Since $\varphi$ is positive on $\cB_{\leq n}$ (by the induction hypothesis),
we have
\begin{align*}0&\leq \varphi(T_k^\ast a T_k)
= \tau(\rho_k(a_{\emptyset, \emptyset})) + \sum_{\n \in \fN_{\leq n}} e^{-\beta x_\n}\tau(a_{k\star\n, k\star\n}).
\end{align*}
This implies
\begin{align*}0&\leq \sum_{k\in \IN} e^{-\beta x_k}\left( \tau(\rho_k(a_{\emptyset, \emptyset})) + \sum_{\n \in \fN_{\leq n}} e^{-\beta x_\n}\tau(a_{k\star\n, k\star\n})\right)\\
&=\sum_{k\in \IN} e^{-\beta x_k} \tau(\rho_k(a_{\emptyset, \emptyset})) + \sum_{\n \in \fN_{\leq n+1}\setminus\{\emptyset\}} e^{-\beta x_\n}\tau(a_{\n, \n}).
\end{align*}
As $\tau$ satisfies the inequality ($\clubsuit$-$\beta$) and $a_{\emptyset, \emptyset}= E_A(a) \in {\rm Ped}(A)_+$, the right hand side value is not greater than
\[\tau(a_{\emptyset, \emptyset}) + \sum_{\n \in \fN_{\leq n+1}\setminus\{\emptyset\}} e^{-\beta x_\n}\tau(a_{\n, \n})=\varphi(a).\]
This proves the positivity of $\varphi|_{\cB_{\leq n+1}}$.
\end{proof}
\begin{Rem}\label{Rem:gauge}
In fact the proof of Theorem \ref{Thm:KMSc}
gives a classification of the \emph{gauge invariant} $\gamma_{\fx}$-$\beta$-KMS weights
without any assumptions on $\fx$:
We only use the assumption on $\fx$ to conclude that any $\gamma_\fx$-$\beta$-KMS weight
is automatically gauge invariant.
\end{Rem}

If the members of $\fr$ generate a free semigroup consisting of sufficiently outer $\ast$-endomorphisms,
then in fact Theorem \ref{Thm:KMSc} gives a complete description of the trace space of $\T$.
As an example of such a phenomenon, we record the next theorem.
Fix a free basis $(s_n)_{n\in \IN}$ of $\IF_\infty$.
Let $\sigma\colon \IF_\infty \acts \cZ^{\otimes \IF_{\infty}}$
denote the Bernoulli shift action.
\begin{Thm}\label{Thm:Ttrace}
Let $A_0$ be a \Cs-algebra and $\fr_0=(\rho_{0, n})_{n\in \IN}$ be a sequence of non-degenerate $\ast$-endomorphisms on $A_0$
satisfying conditions $(\diamondsuit 1)$ and $(\diamondsuit 2)$.
On the \Cs-algebra $A:=A_0 \otimes \cZ^{\otimes \IF_\infty}$,
consider the sequence ${\fr}=(\rho_n)_{n\in \IN}:=(\rho_{0, n} \otimes \sigma_{s_n})_{n\in \IN}$ of $\ast$-endomorphisms.
Then any tracial weight on $\T$ is gauge invariant.
\end{Thm}

\begin{Rem}\label{Rem:Ttrace}
We keep the setting of Theorem \ref{Thm:Ttrace}.
\begin{enumerate}
\item By Theorems \ref{Thm:KMSc} and \ref{Thm:Ttrace} and Remark \ref{Rem:gauge}, we obtain a bijection from the trace space ${\rm T}(\T)$ onto the space
\[{\rm T}_\fr(A):=\left\{\tau\in {\rm T}(A): \sum_{n\in \IN} \tau(\rho_n(a)) \leq \tau(a) {\rm~for~all~}a\in {\rm Ped}(A)_+ \right \}.\]
The bijection is given by the restriction:
\[{\rm T}(\T)\ni \tau \mapsto \tau|_{A_+}\in {\rm T}_\fr(A).\]

\item
The statement is strong enough for our main motivation ---construction of amenable actions on stably finite simple \Cs-algebras with prescribed (or at laest determinable) invariants---,
because taking the tensor product with 
$\cZ^{\otimes \IF_{\infty}} \cong \cZ$ does not change (the dynamical system on) the trace space and the (equivariant) Kasparov class,
and $\cZ$-stability is a standard requirement in classification theory of simple \Cs-algebras.

\item The resulting sequence $\fr$ satisfies conditions $(\diamondsuit 1)$ and $(\diamondsuit 2)$.
\item In the proof below, the Bernoulli shift action $\sigma \colon \IF_\infty \acts \cZ^{\otimes \IF_\infty}$ in the statement can be replaced by any centrally $\IF_\infty$-free action
 (that is, the identity inclusion is centrally $\IF_\infty$-free in the sense of \cite{SuzCMP})
on a separable \Cs-algebra.
\item Thanks to Theorem \ref{Thm:KMS}, the proof below in fact works
for all $\gamma_\fx$-KMS weights on $\T$,
for any $\fx\in \IR^\IN$.
Thus, for the present $(A, \fr)$, the statement of Theorem \ref{Thm:KMSc} holds true for all $\fx\in \IR^\IN$ (cf.~Remark \ref{Rem:gauge}).
\item The group $\IF_\infty$ and the sequence $(s_n)_{n\in \IN}$ in the construction can be replaced by
any countable group $\Gamma$ and a sequence $(t_n)_{n\in \IN} \subset \Gamma$
which generates a free semigroup in $\Gamma$. In particular, one can replace $\IF_\infty$ by a solvable group.
\end{enumerate}
\end{Rem}
\begin{proof}[Proof of Theorem \ref{Thm:Ttrace}]
Take any tracial weight $\tau$ on $\T$.
To show the gauge invariance of $\tau$, by Proposition \ref{Prop:trace} (and the tracial condition and the relations among $A$ and $T_\m$, $\m\in \fN$), it suffices to show the equality \[\tau(a T_\n^\ast)=0\quad{\rm~ for~} a\in {\rm Ped}(A_0)\odot \mathcal{Z}^{\odot \mathbb{F}_\infty} \subset A,~ \n \in \fN\setminus \{\emptyset\}.\]
Here `$\odot$' denotes the algebraic tensor product.
Take any $a\in {\rm Ped}(A_0)\odot \mathcal{Z}^{\odot \mathbb{F}_\infty} \subset A$ and $\n \in \fN\setminus \{\emptyset\}$.
By the linearity of $\tau$ (and the definition of the Pedersen ideal), we may assume that there exists $e\in {\rm Ped}(A_0)$ with $ae=ea=a$.
Observe that $(s_n)_{n\in \IN}$ generates a free semigroup in $\IF_\infty$,
hence by the observation in Example 4.14 of \cite{SuzCMP}, one can find a sequence $b_1, \ldots, b_r \in (\mathcal{Z}^{\otimes \mathbb{F}_\infty})_+ \subset \cM(A)$
satisfying
\[\sum_{i=1}^r b_i^2=1,\quad \sum_{i=1}^r b_i a \rho_{\n}(b_i) = 0.\]
By using this sequence, we obtain
\[\tau(a T_\n^\ast) = \sum_{i=1}^r \tau(b_ie b_i a T_\n^\ast)=\sum_{i=1}^r \tau(b_i a \rho_{\n}(b_i) T_\n^\ast e)=0.\]
\end{proof}

\section{Finite rank analogue}\label{section:fr}
In this section, we study the finite rank analogue of our previous construction.
More precisely, for a finite set $S$ with $|S|\geq 2$, we consider a family $\fu:=(\upsilon_s)_{s\in S}$ of injective, non-degenerate $\ast$-endomorphisms on $A$
satisfying the following condition.
\begin{description}
\item[($\heartsuit$)] There is no proper ideal $I$ of $A$
satisfying $\upsilon_s(I)\subset I$ for all $s\in S$.
\end{description}
We also fix $\fy:=( y_s)_{s\in S} \in \IR^S$.
Then, analogous to the infinite rank case $\E$ (see Section \ref{section:corr}), we define the \Cs-correspondence $\cF$ over $A$ and the unitary representation
$u_\fy \colon \IR \rightarrow \cU_A(\cF)$ as follows.
\begin{itemize}
\item $\cF= \ell^2(S)\otimes A$ as a right Hilbert \Cs-$A$-module.
\item The left $A$-action ${\pi_\fu} \colon A \rightarrow \IB_A(\cF)$ is given by
\[{\pi_\fu}(a)(\delta_s \otimes b):=\delta_s\otimes \upsilon_s(a)b.\]
\item $u_{\fy, t}(\delta_s\otimes a):=e^{iy_st} \delta_s\otimes a$.
\end{itemize}
Let $\gamma_\fy \colon \IR \acts \mathcal{T}_{\cF}$ and $\bar{\gamma}_\fy\colon \IR \acts \C$ denote the quasi-free flows
arising from $u_\fy$.

For the \Cs-correspondence $\cF$,
we have ${\pi_\fu}(A) \subset \IK_A(\cF)$.
Hence the Toeplitz--Pimsner algebra $\mathcal{T}_{\cF}$ and the reduced crossed product $\mathcal{T}_{\cF} \rtimes_{\gamma_\fy} \IR$ can never be simple.
Instead, we consider the Cuntz--Pimsner algebra $\C$.
After appropriate modifications, we obtain analogous results for $\C$ and
$\bar{\gamma}_\fy$.
Similar to the difference between Cuntz algebras $\mathcal{O}_n$ ($2\leq n<\infty$) and $\mathcal{O}_\infty$,
the structures of $\C$, $\C\rtimes_{\bar{\gamma}_\fy}\IR$
have some different features from the infinite rank variants $\T$($=\mathcal{O}_{\E}$), $\T\rtimes_{{\gamma}_\fx} \IR$.

We define $S^0:=\{\emptyset\}$ and $\fS:=\bigsqcup_{n\in \IN} S^n$.
We use the notations analogous to the infinite rank case introduced in Section \ref{section:corr} without explanations.
For instance, for $\fs=(s_1, \ldots, s_n)\in \fS$, we put
\[\upsilon_\fs:=\upsilon_{s_n}\circ \upsilon_{s_{n-1}}\circ \cdots \circ \upsilon_{s_1}.\]
For $\fs\in \fS$, denote by $S_\fs\in \cM(\C)$ the quotient image of $T_\fs \in \cM(\mathcal{T}_{\cF})$.

We note that, for the present $\cF$, the kernel of the quotient map $\mathcal{T}_{\cF} \rightarrow \C$
is the ideal generated by $\{a-\sum_{s\in S} T_s \upsilon_s(a) T_s^\ast:a\in A\}$.

We first give a classification result of KMS weights for $\bar{\gamma}_\fy$ (cf.~Theorem \ref{Thm:KMSc}).

\begin{Thm}\label{Thm:KMSc2}
Assume that $D(\IT^S)\subset K_\fy$.
Then, for every $\beta\in \IR$, there is a bijective correspondence
between $\bar{\gamma}_\fy$-$\beta$-KMS weights on $\C$ and
tracial weights $\tau$ on $A$ satisfying the equality
\[(\spadesuit\mathchar`-\beta)\quad \sum_{s\in S} e^{-\beta y_s}\tau(\upsilon_s(a))= \tau(a) \quad {\rm~for~all~}a\in \Ped(A)_+.\]
The bijection is given by
\[\varphi \mapsto \tau:=\varphi|_{A_+}.\]
\end{Thm}
\begin{proof}
The proof of Theorem \ref{Thm:KMSc} works for the flow $\gamma_\fy \colon \IR \acts \mathcal{T}_{\cF}$ instead of $\gamma_\fx$ therein.
Therefore it suffices to show that a $\gamma_\fy$-$\beta$-KMS weight $\varphi$ on $\mathcal{T}_{\cF}$
factors through $\C$ if and only if
the tracial weight $\tau :=\varphi|_{A_+}$ satisfies the equality ($\spadesuit\mathchar`-\beta$) (which is stronger than the inequality corresponding to ($\clubsuit$-$\beta$) therein).

For any $\gamma_\fy$-$\beta$-KMS weight $\varphi$ on $\mathcal{T}_{\cF}$ and $a\in \Ped(A)$, with $\tau:=\varphi|_{A_+}$,
the KMS condition implies the equality
\[\varphi(a-\sum_{s\in S} T_s \upsilon_s(a)T_s^\ast)=\tau(a)-\sum_{s\in S} e^{-\beta y_s}\tau(\upsilon_s(a)).\] 
Since $a-\sum_{s\in S} T_s \upsilon_s(a)T_s^\ast$ is contained in the kernel of the quotient map $\mathcal{T}_{\cF} \rightarrow \C$,
this shows the necessity of the equation ($\spadesuit\mathchar`-\beta$) to $\varphi$ factoring through $\C$.

Conversely, assume that $\tau:=\varphi|_{A_+}$ satisfies the equation ($\spadesuit\mathchar`-\beta$).
Then, for any $a, b \in \mathcal{A}_{\gamma_\fy}$ and any $c\in \Ped(A)_+$, since 
\[d:=c-\sum_{s\in S} T_s \upsilon_s(c)T_s^\ast =(1-\sum_{s\in S} T_s T_s^\ast) c(1-\sum_{s\in S} T_s T_s^\ast) \geq 0,\quad \varphi(d)=0,\]
the Cauchy--Schwarz inequality and the KMS condition yield
\begin{align*}
|\varphi(adb)|^2
&\leq \varphi(ad a^\ast)\varphi(b^\ast d b)\\
&=\varphi(d^{1/2} \gamma_{\fy, -i\beta/2}(a^\ast) \gamma_{\fy, -i\beta/2}(a^\ast)^\ast d^{1/2})\varphi(d^{1/2} \gamma_{\fy, -i\beta/2}(b) \gamma_{\fy, -i\beta/2}(b)^\ast d^{1/2})\\
&\leq \|\gamma_{\fy, -i\beta/2}(a^\ast)\|^2 \|\gamma_{\fy, -i\beta/2}(b)\|^2 \varphi(d)^2=0.
\end{align*}
Since $\varphi$ is lower semi-continuous, this shows that $\varphi$ factors through $\C$, because the set
\[ \mathcal{A}_{\gamma_\fy}\cdot \left\{x-\sum_{s\in S} T_s \upsilon_s(x)T_s^\ast: x\in \Ped(A)\right\} \cdot \mathcal{A}_{\gamma_\fy}\]
spans a dense $\ast$-subalgebra of the kernel of the quotient map $\mathcal{T}_{\cF}\rightarrow \C$.
\end{proof}
Analogues of Theorem \ref{Thm:Ttrace} and Remark \ref{Rem:Ttrace}
are also valid in the finite rank case.
The proof is the same as that of Theorem \ref{Thm:Ttrace},
hence we only record the statement.
Let $\IF_S$ denote the free group generated by $S$,
and let $\sigma \colon \IF_S \acts \cZ^{\otimes \IF_S}$ denote the Bernoulli shift action.
\begin{Thm}
Let $A_0$ be a \Cs-algebra and $\fu_0=(\upsilon_{0, s})_{s\in S}$ be a family of non-degenerate, injective $\ast$-endomorphisms on $A_0$ satisfying condition $(\heartsuit)$.
On the \Cs-algebra $A:=A_0 \otimes \cZ^{\otimes \IF_S}$,
consider the family ${\fu}=(\upsilon_s)_{s \in S}:=(\upsilon_{0, s} \otimes \sigma_{s})_{s\in S}$ of $\ast$-endomorphisms.
Then any tracial weight on $\C$ is gauge invariant.
\end{Thm}

We next show the simplicity of $\C$.
The proof of the simplicity of the Toeplitz--Pimsner algebra for $\E$ (Theorem \ref{Thm:Tsimple})
does not work for $\cF$ because $B_n$ is no longer essential in $B_{\leq n}$ for $n\geq 1$.
Indeed for any $a\in A$ and any $n\in \IN$, one has
\[\left(a - \sum_{\fs\in S^n} T_\fs \upsilon_\fs(a) T_\fs^\ast \right)\cdot B_n=\{0\}=B_n\cdot \left(a - \sum_{\fs\in S^n} T_\fs \upsilon_\fs(a) T_\fs^\ast \right),\]
while $a - \sum_{\fs\in S^n} T_\fs \upsilon_\fs(a) T_\fs^\ast\neq 0$ unless $a\neq 0$.

\begin{Lem}\label{Lem:giidf}
There is no proper ideal $I$ of $\C^{\bar{\sigma}}$ satisfying
$S_\fs^\ast \cdot I \cdot S_\fs\subset I$ for all $\fs\in \fS$. 
\end{Lem}
\begin{proof}
Since $a=\sum_{s\in S} S_s \upsilon_s(a) S_s^\ast$ for all $a\in A$, we have
$B_n \subset B_{n+1}$ for all $n\in \IN$.
Take a nonzero ideal $I \subset \C^{\bar{\sigma}}$ satisfying $S_\fs^\ast \cdot I \cdot S_\fs\subset I$ for all $\fs\in \fS$. 
Since the increasing union $\bigcup_{n\in \IN} B_n$ is dense in $\C^{\bar{\sigma}}$,
we have $n\in \IN$ with $I\cap B_n \neq \{0\}$.
Let $\Theta_n\colon B_n \rightarrow \IK(\ell^2(S^n))\otimes A$
denote the canonical $\ast$-isomorphism.
Choose an ideal $J_n\subset A$ with
\[\Theta_n(I\cap B_{n})=\IK(\ell^2(S^n))\otimes J_n.\]
Observe that for any $\fs \in S^n$ and any $a\in A$,
one has
\[S_{\fs}^\ast\Theta_n^{-1}(\delta_\fs \otimes a)S_{\fs}=a.\]
This yields
\[\{0\} \neq J_n \subset I \cap A,\]
whence
\[\upsilon_{\fs}(J_n) = S_\fs^\ast \cdot J_n \cdot S_\fs \subset S_\fs^\ast \cdot I \cdot S_\fs \subset I {\rm~for~all~}\fs\in \fS.\]
By condition $(\heartsuit)$ of $\fu$, we conclude $A\subset I$, whence $I=\C^{\bar{\sigma}}$.
\end{proof}
\begin{Rem}
Assume that $\fu$ satisfies the following stronger condition:
For any nonzero ideal $I$ of $A$, there is $n\in \IN$
with $\cI(\bigcup_{\fs\in S^n} \upsilon_\fs(I); A) =A$.
Then one can conclude the simplicity of $\C^{\bar{\sigma}}$. 
\end{Rem}

\begin{Thm}\label{Thm:Csimple}
The Cuntz--Pimsner algebra $\C$ is simple.
\end{Thm}
\begin{proof}
Let $I$ be a nonzero ideal of $\C$.
In the analogous ways to the proof of Theorem \ref{Thm:Tsimple},
we define the subspaces $V_n, V_{n, \leq m}\subset \C$ for $n\in \IZ$ and $m\geq |n|$.
Then by using Lemma \ref{Lem:giidf} instead of Lemma \ref{Lem:giid} in the proof of Theorem \ref{Thm:Tsimple},
one can find $x\in I$ satisfying $\|E^{\bar{\sigma}}(x)\|=1$, $E^{\bar{\sigma}}(x)\in A$.

Take $N\in \IN$, $ x_n\in V_{n, \leq N}$; $n=-N,\cdots, N$ satisfying
\[x\approx_{1/2}\sum_{n=-N}^N x_n,\quad x_0=E^{\bar{\sigma}}(x).\]
Take $\fs=(s_1, s_2, \ldots, s_{2N})\in S^{2N}$ such that the subsequences
\[(s_1, \ldots, s_{N}), (s_2, \ldots, s_{N+1}), \ldots, (s_{N+1}, \ldots, s_{2N})\in S^{N}\]
are mutually distinct.
(For instance, take two distinct elements $s, t \in S$.
Define $s_1=s_2=\cdots = s_{N}=s$, $s_{N+1}=\cdots = s_{2N}=t$.
Then the resulting sequence $\fs=(s_1, s_2, \ldots, s_{2N})$ has the desired property.)
We then have
\[S_\fs^\ast\cdot V_{n, \leq N} \cdot S_\fs =\{0\} \quad {\rm~for~all~}n\in \IZ {\rm~with~}1 \leq |n|\leq N.\]
This implies
\[I \ni S_\fs^\ast x S_\fs \approx_{1/2} \upsilon_\fs(x_0) \in \C^{\bar{\sigma}},\]
whence $I \cap \C^{\bar{\sigma}}\neq \{0\}$.
This together with Lemma \ref{Lem:giidf} implies
$I=\C$.
\end{proof}

The proof of Lemma \ref{Lem:idBn} also works for the finite rank case.
Therefore the rest of the proof of the following theorem is the same as the proof of Theorem \ref{Thm:Rsimple},
after using Theorem \ref{Thm:Csimple} instead of Theorem \ref{Thm:Tsimple}.
We thus omit the proof.

\begin{Thm}\label{Thm:RsimpleC}
Assume that $\fy=(y_s)_{s\in S}\in \IR^S$ satisfies the following two conditions.
\begin{itemize}
\item $D(\IT^S)\subset K_\fy$.
\item For any nonzero ideal $I$ of $A$, there is a nonzero ideal $J$ of $A$ such that the set
$\{y_\fs: \fs\in \fS, J \subset \cI(\rho_{\fs}(I); A)\}$ is dense in $\IR$.
\end{itemize}
Then the reduced crossed product $\C\rtimes_{\bar{\gamma}_\fy} \IR$ is simple.
\end{Thm}

\begin{Exm}
When $\fy=(y_s)_{s\in S} \in (\IR\setminus\{0\})^S$,
then for any $(c_s)_{s\in S} \in (0, \infty)^S$ and $\beta_0 \in \IR$
satisfying
\[\sum_{s\in S} c_s e^{-\beta_0 y_s}\leq 1,\]
the Intermediate Value Theorem shows the existence of $\beta\in \IR$ with
\[\sum_{s\in S} c_s e^{-\beta y_s}=1\]

Thus, if $A$ admits a tracial weight $\tau$ and positive constants $c_s$; $s\in S$
satisfying $\tau \circ \upsilon_s=c_s \cdot \tau$ for all $s\in S$, and $c_{s_1}+c_{s_2}<1$ for two distinct elements $s_1, s_2\in S$,
then one can find $\fy=(y_n)_{s\in S}\in \IR^S$ which is linearly independent over $\mathbb{Q}$, $y_{s_1} y_{s_2}<0$,
and satisfies the equation ($\spadesuit\mathchar`-\beta$) in Theorem \ref{Thm:KMSc2} for $\tau$ and some $\beta \in \IR$:
\[\sum_{s\in S} e^{-\beta y_s}\tau(\upsilon_s(a))= \tau(a) \quad {\rm~for~all~}a\in \Ped(A)_+\]

If we additionally assume that
$I \subset \cI(\rho_{s_1}(I); A), \cI(\rho_{s_2}(I); A)$ for all ideals $I$ of $A$,
then the set $\{y_\fs: \fs\in \fS, I \subset \cI(\rho_{\fs}(I); A)\}$ contains the subsemigroup of $\IR$ generated by $y_{s_1}$ and $y_{s_2}$,
which is dense in $\IR$.
Therefore, by Theorems \ref{Thm:KMSc2} and \ref{Thm:RsimpleC}, for such $\fy$, the reduced crossed product $\C \rtimes_{\bar{\gamma}_\fy} \IR$
is stably finite and simple.
On the one hand, in the case that $\sum_{s\in S} c_s \neq 1$, the above $\beta$ must be nonzero.
Hence the simple \Cs-algebra $\C \rtimes_{\bar{\gamma}_\fy} \IR$ is in fact stably projectionless, by Remark \ref{Rem:spl}.
On the other hand, in the case that $\sum_{s\in S} c_s =1$, the \Cs-algebra $\C$, which is simple by Theorem \ref{Thm:Csimple}, is also stably finite by Theorem \ref{Thm:KMSc2}.
\end{Exm}

\section{Dynamical version of the constructions and their properties}\label{section:equiv}
In the rest of this article, let $G$ be a locally compact second countable group.
Now we would like to apply our constructions to $G$-\Cs-algebras $A$.
Denote by $\theta \colon G \curvearrowright A$ the equipped action.
We first consider the infinite rank case (Section \ref{section:corr}).
To respect the given group action in the construction, we further require all members $\rho_n$ of $\fr$ to be $G$-equivariant.
We then equip ${\mathcal{E}_{\fr}}=\ell^2(\IN) \otimes A$ with the $G$-action $\check{\theta}:=\id_{\ell^2(\IN)} \otimes \theta$.
This makes ${\mathcal{E}_{\fr}}$ a $G$-\Cs-correspondence over $A$,
that is, the action satisfies the equalities
\[\check{\theta}_g(\pi_{\fr}(a)\xi)= \pi_{\fr}(\theta_g(a))\check{\theta}_g(\xi),\quad \check{\theta}_g(\xi a)=\check{\theta}_g(\xi)\theta_g(a),\quad
\ip{\check{\theta}_g(\xi), \check{\theta}_g(\eta)}=\theta_g(\ip{\xi, \eta})\]
for $g\in G$, $a\in A$, $\xi, \eta \in \E$.
By the universality of $\T$, one can define an action $\Theta \colon G \curvearrowright \T$
by the formula
\[\Theta_g(T_{\delta_n\otimes a}):=T_{\delta_n \otimes {\theta_g}(a)}\quad{\rm~for~}g\in G, n\in \IN, a\in A.\]
Clearly $\Theta$ commutes with the flow $\gamma_\fx$ for every $\fx\in \IR^\IN$.
Hence $\Theta$ induces the action
\[\Xi \colon G \curvearrowright \T \rtimes_{\gamma_\fx} \IR\]
given by
\[\Xi_g(f)(t):=\Theta_g(f(t))\quad{\rm~for~}g\in G, f\in C_c(\IR, \T), t\in \IR.\]
We first show that these constructions preserve amenability of the actions.
In the proof below, Theorem 5.1 of \cite{OS} plays a fundamental role.
\begin{Thm}\label{Thm:ame}
The actions
\[\Theta\colon G \acts \T, \quad \Xi \colon G \curvearrowright \T \rtimes_{\gamma_\fx} \IR\]
are amenable if and only if the original action
$\theta\colon G \acts A$ is amenable.
\end{Thm}
\begin{proof}
We only prove the statement for $\Xi$.
The statement for $\Theta$ follows from the proof of Theorem 6.1 in \cite{OS} (or the proof below).

($\Longrightarrow$): Obviously, the canonical inclusion $A \subset \T$ and the canonical conditional expectation $E_A \colon \T \rightarrow A$ are $G$-equivariant.
Hence we have a $G$-equivariant inclusion $A \otimes \Cso(\IR) \subset \T \rtimes_{\gamma_\fx}\IR$ with a $G$-equivariant conditional expectation.
Clearly the action $\theta \otimes \id_{\Cso(\IR)}\colon G \acts A\otimes \Cso(\IR)$ is amenable if and only if $\theta$ is amenable.
Thus the amenability of $\Xi$ implies that of $\theta$ by Proposition 3.7 of \cite{OS}.

($\Longleftarrow$):
Assume that $\theta$ is amenable.
Let $\tilde{\sigma} \colon \IT \acts \T \rtimes_{\gamma_\fx} \IR$ denote the action
induced from the gauge action.
To show the amenability of $\Xi$,
by Theorem 5.1 of \cite{OS},
 it suffices to show
the amenability of the subsystem
\[G\acts (\T \rtimes_{\gamma_\fx} \IR)^{\tilde{\sigma}}=E^{\tilde{\sigma}}(\T \rtimes_{\gamma_\fx} \IR) = \T^{\sigma} \rtimes_{\gamma_\fx} \IR.\]

Observe that for any $n\in \IN$,
the $\ast$-isomorphism
\[B_n \rtimes_{\gamma_\fx} \IR \cong \IK(\ell^2(\IN^n)) \otimes A \otimes \Cso(\IR)\]
in the proof of Lemma \ref{Lem:idBn}
is $G$-equivariant
when we equip $\IK(\ell^2(\IN^n)) \otimes A \otimes \Cso(\IR)$ with the $G$-action $\id_{\IK(\ell^2(\IN^n))}\otimes \theta \otimes \id_{{\rm C}^\ast(\IR)}$.
Therefore the action
$G \curvearrowright B_n \rtimes_{\gamma_\fx} \IR$
is amenable for all $n\in \IN$.
Observe that $(\T \rtimes_{\gamma_\fx} \IR)^{\tilde{\sigma}}$
is the closure of the increasing union $\bigcup_{n \in \IN} (B_{\leq n} \rtimes_{\gamma_\fx} \IR)$.
The entries of the union fulfil the canonical short exact sequences
\[\{0\} \rightarrow B_{n+1} \rtimes_{\gamma_\fx} \IR \rightarrow B_{\leq n+1} \rtimes_{\gamma_\fx} \IR \rightarrow B_{\leq n} \rtimes_{\gamma_\fx} \IR\rightarrow \{0\},\]
which are clearly $G$-equivariant.
Thus Proposition 3.7 of \cite{OS} implies the amenability of $G\acts (\T \rtimes_{\gamma_\fx} \IR)^{\tilde{\sigma}}$.
\end{proof}

Similar to the infinite rank case, in the finite rank case (Section \ref{section:fr}), under the additional assumption that
the members of $\fu$ are $G$-equivariant, we have induced actions
\[\bar{\Theta}\colon G \acts \C,\quad \bar{\Xi}\colon G \acts \C\rtimes_{\bar{\gamma}_\fy} \IR.\]

Amenability in the finite rank case is essentially the same as the infinite rank case (Theorem \ref{Thm:ame}).
We therefore omit a proof, and only record the statement.
\begin{Thm}
The actions
\[\bar{\Theta}\colon G \acts \C,\quad \bar{\Xi}\colon G \acts \C \rtimes_{\bar{\gamma}_\fy} \IR\]
are amenable if the original action $\theta \colon G \acts A$
is amenable.
\end{Thm}

We end this section by recording the following advantage of our present construction.
For a $G$-\Cs-algebra $A$,
we regard the suspension $SA:=C_0(\IR)\otimes A$
as the $G$-\Cs-algebra by equipping with the $G$-action $\id_{C_0(\IR)} \otimes \theta$.
\begin{Thm}\label{Thm:Kas}
Let $A$ be a separable $G$-\Cs-algebra.
Then for any sequence $\fr$ of non-degenerate $G$-equivariant $\ast$-endomorphisms on $A$
and any $\fx\in \IR^\IN$,
the $G$-\Cs-algebras $\T$, $\T \rtimes_{\gamma_\fx} \IR$
are ${\rm KK}^G$-equivalent to $A$, $SA$ respectively.
\end{Thm}
\begin{proof}
The statement for $\T$ follows because the proof of Pimsner's KK-equivalence theorem \cite{Pim}
works in the equivariant setting (see \cite{Mey}).

Similarly, the latter statement follows because the proof of Connes's Thom isomorphism theorem in \cite{FS}
works in the equivariant setting.
(This is a special case of the strong Baum--Connes conjecture \cite{MN} (cf.~Proposition 4.14 of \cite{Sza}),
because any flow on a \Cs-algebra is homotopy equivalent to the trivial flow.)
\end{proof}

\section{Examples and applications}
We close this article by giving a few concrete applications of our main theorems.
We believe that these results have their own importance, and furthermore, the proofs presented below illustrate flexibility of our new constructions.
We note that some ingredients of the constructions below are the same as ingredients of the construction in \cite{Suzsf} in the finitely generated discrete group case.

The next theorem generalizes and enhances the main theorem of \cite{Suzsf} (Theorem A).
Let $\mathcal{Z}_0$ denote the stably projectionless analogue of the Jiang--Su algebra \cite{GL}.
Note that for any $c \in (0, \infty)$, there is a $\ast$-automorphism $\alpha$ on $\cZ_0 \otimes \IK$
with $\tau_0 \circ \alpha =c \cdot \tau_0$.
Here $\tau_0$ denotes a tracial weight on $\cZ_0 \otimes \IK$, which is unique up to scalar multiple.
\begin{Thm}\label{Thm:Exm}
For any locally compact second countable group $G$,
there is an amenable action $G\acts A$ on a classifiable simple stably projectionless \Cs-algebra $A$
which is ${\rm KK}^G$-equivalent to the left translation action $G \acts C_0(G)$.
\end{Thm}
\begin{proof}
Take a sequence $(g_n)_{n \in \IN} \subset G\times \IR$ which generates a dense subsemigroup of $G\times \IR$ and satisfies $g_0=g_1=(e, 0)$.
We equip $A:=C_0(G\times \mathbb{R})\otimes \cZ_0 \otimes \IK$ with the left translation $G$-action ${\rm L}\otimes \id_{\cZ_0\otimes \IK}$.
Choose a sequence $(c_n)_{n\in \IN} \in (0, \infty)^\IN$ and a sequence $(\alpha_n)_{n\in \IN}$ of $\ast$-automorphisms on $\cZ_0\otimes \IK$
with
\[c_0 + c_1 < 1, \quad \tau_0 \circ \alpha_n =c_n \cdot \tau_0\quad{\rm~for~each~}n\in \IN.\]
Define $\fr:=({\rm R}_{g_n}\otimes \alpha_n)_{n \in \IN}$,
where ${\rm R}\colon G\times \IR \acts C_0(G\times \IR)$
denotes the right translation action.
Clearly $\fr$ satisfies conditions $(\diamondsuit 1)$ and $(\diamondsuit 2)$ in Section \ref{section:corr}.
It is also clear that the members of $\fr$ are $G$-equivariant.
Let $\tau \colon A_+ \rightarrow [0, \infty]$
denote the tensor product tracial weight of the right Haar measure on $G \times \IR$ and $\tau_0$.
Then one has
$\tau \circ \rho_n = c_n \cdot \tau$ for each $n\in \IN$.
By the observation in Example \ref{Exm:KMST},
for any $\beta \in \IR \setminus\{0\}$, there exists $\fx \in \IR^\IN$
satisfying the inequality ($\clubsuit$-$\beta$) in Theorem \ref{Thm:KMSc} for the present $\tau$
and the assumption of Theorem \ref{Thm:Rsimple}.
For any such $\beta$ and $\fx$,
the induced action
$\Xi\colon G \acts \T\rtimes_{\gamma_\fx} \IR$
gives an amenable action on a simple, separable, nuclear, and stably projectionless \Cs-algebra.
By Theorem \ref{Thm:Kas},
$\Xi$ is KK$^G$-equivalent to $C_0(G\times \IR^2)$.
Hence it is in fact KK$^G$-equivalent to $C_0(G)$ by the Bott periodicity.
In particular $\T \rtimes_{\gamma_\fx} \IR$ satisfies the universal coefficient theorem \cite{RS}.

Now the diagonal action of $\Xi$ and the trivial action on $\mathcal{Z}_0$
gives the desired action.
\end{proof}
\begin{Rem}\label{Rem:Exm}
In the above proof, by replacing $G\times \IR$ by $G$ and by requiring the sequence $(c_n)_{n\in \IN} \in (0, \infty)^\IN$ to satisfy
$\sum_{n\in \IN} c_n \leq 1$ instead of $c_0 + c_1 <1$, one concludes that
the induced action $G \acts \T\otimes \cZ_0$ also gives the desired action.
\end{Rem}

The next theorem may be seen as a stably projectionless analogue of the purely infinite result, Corollary 6.4 in \cite{BO}, in the free group case
(which was originally shown for the free groups in the proof of Theorem 5.1 in \cite{SuzCMP} based on the construction in \cite{Suzeq}).
This means that, at least on the level of equivariant Kasparov theory,
there is no obstruction on amenable actions of free groups on classifiable simple stably projectionless \Cs-algebras.
 \begin{Thm}
 Any countable free group $\Gamma$
 admits an amenable action on a classifiable simple stably projectionless \Cs-algebra
 which is ${\rm KK}^\Gamma$-equivalent to $\IC$. 
 \end{Thm}
 \begin{proof}
Consider the universal covering (Lie) group $G$ of ${\rm SL}(2, \mathbb{R})$.
The Iwasawa decomposition gives a homeomorphism from $G$ onto $\IR^3$.
Thus, by the Bott periodicity, $C_0(G \times \IR)$ is KK-equivalent to $\IC$.
Observe that $G$ contains countable free groups as closed subgroups.
(Indeed any homomorphism lifting of an isomorphism $\mathbb{F}_d \rightarrow {\rm SL}(2, \mathbb{R})$ onto a closed subgroup gives the desired embedding.)
For a given free group $\Gamma$, we fix such an embedding $\Gamma < G$.

We now apply Theorem \ref{Thm:Exm} (or Remark \ref{Rem:Exm}) to $G\times \IR$.
Denote by $\alpha \colon \Gamma \acts A$ the restriction of the resulting action to the closed free subgroup $\Gamma<G \times \IR$.
Then, as $G\times \IR$ is path connected and $\Gamma$ is a free group,
the left translation action $\mathrm{L}|_{\Gamma} \colon \Gamma \acts C_0(G\times \IR)$
is homotopic to the trivial action $\Gamma \acts C_0(G \times \IR)$.
This together with Theorem \ref{Thm:Kas} shows that
the $\Gamma$-\Cs-algebra $(A, \alpha)$
is KK$^\Gamma$-equivalent to $\IC$.
(For a proof of this implication, see e.g., Proposition 4.14 of \cite{Sza},
whose proof in fact works for any locally compact second countable groups with the Haagerup property without non-trivial compact subgroups by \cite{HK} and \cite{MN}.
In particular it works for free groups.)
Since ${\mathrm L}|_\Gamma$ is amenable, so is $\alpha$ by Theorem \ref{Thm:ame}.
Thus $\alpha$ gives the desired action.
 \end{proof}
 We also give concrete examples of amenable actions $G \acts \T$, $G \acts \C$
 such that the \Cs-algebras $\T$ and $\C$ admit ($G$-invariant) tracial weights, by using some ingredients of \cite{Suzsf}.
 \begin{Thm}
 There is a $G$-\Cs-algebra $A$ and a sequence $\fr$ of $G$-$\ast$-endomorphisms on $A$ as in Section \ref{section:corr}
 such that the induced action $G \acts \T$ is amenable and $\T$ admits a tracial weight.
When $G$ is topologically finitely generated,
 one can find a finite family $\fu$ of $G$-$\ast$-endomorphisms on a $G$-\Cs-algebra with the analogous properties.
 \end{Thm}
 \begin{proof}
We only show the latter statement. The first statement can be shown in a similar way.
Take a finite sequence $(g_n)_{n=1}^N$ in $G$ which generates a dense subsemigroup of $G$ and satisfies $N\geq 2$.
Take a simple \Cs-algebra $B$, a tracial weight $\tau_1$ on $B$,
non-degenerate $\ast$-endomorphisms $\rho_n \colon B \rightarrow B$; $n=1, \ldots, N$, and $r_1, \ldots, r_N\in (0, \infty)$
which satisfy
\[\tau_1\circ \rho_n =r_n\cdot \tau_1\quad {\rm~for~}n=1, \ldots, N, \quad \sum_{n=1}^N r_n \Delta_G(g_n) =1.\]
Here $\Delta_G \colon G \rightarrow (0, \infty)$ denotes the modular function on $G$.
Put $A:=C_0(G) \otimes B$.
We equip $A$ with the left translation $G$-action ${\rm L}\otimes \id_{B}$.
Put $S:=\{1, \ldots, N\}$.
Define $\fu:=({\rm R}_{g_n} \otimes \rho_n)_{n=1}^N$.
Clearly $\fu$ satisfies condition $(\heartsuit)$, and all members of $\fu$ are $G$-equivariant.
Now let $\tau_2$ denote the tracial weight on $C_0(G)$ defined by the left Haar measure.
Then the tracial weight $\tau:=\tau_2 \otimes \tau_1$ on $A$ is $G$-invariant, and
satisfies the equation $(\spadesuit$-$0)$ in Theorem \ref{Thm:KMSc2} (for all $\fy\in \IR^S$).
Hence $\C$ admits a $G$-invariant tracial weight.

 \end{proof}
 
 As another application of our main results Theorems \ref{Thm:Tsimple} and \ref{Thm:KMSc}, we obtain the following crossed product decomposition results
for the stabilizations of $\mathcal{Z}_0$ (hence all classifiable,  simple, stably projectionless \Cs-algebras by the tensor absorption) and the UHF algebras $U$ of infinite type.
We note that the analogous result for the stabilization of $\mathcal{O}_\infty$ is
a fundamental ingredient of extremely tight \Cs-algebra inclusions around Kirchberg algebras constructed in \cite{SuzMAAN}.
We also remark that these results do not hold true for the \Cs-algebras $\cZ_0$ and $U$ themselves,
because amenable actions of non-amenable groups cannot have an invariant state by Proposition 3.5 of \cite{OS}.
(By the K-theoretical obstruction arising from the Pimsner--Voiculescu exact sequence \cite{PV},
the Cuntz algebra $\mathcal{O}_\infty$ also does not admit the crossed product decompositions, at least over the free groups.)
\begin{Thm}\label{Thm:dec}
For any countable group $\Gamma$, there is an amenable action $\alpha \colon \Gamma \acts A$
on a classifiable, stable, simple, projectionless \Cs-algebra
whose reduced crossed product $A\rtimes_{\alpha} \Gamma$ is
isomorphic to $\mathcal{Z}_0 \otimes \IK$.
\end{Thm}
\begin{proof}
We first consider the \Cs-algebras
\[C:= \mathcal{Z}_0 \otimes \IK \otimes \mathcal{Z}^{\otimes \mathbb{F}_\infty}, \quad B:=c_0(\Gamma) \otimes C.\]
We equip $B$ with the left translation $\Gamma$-action $\theta:={\rm L}\otimes \id_{C}$.

Choose a sequence $(g_n)_{n\in \IN} \subset \Gamma$ which generates $\Gamma$ as a semigroup.
Fix a tracial weight $\tau_0$ on $\mathcal{Z}_0 \otimes \IK$ (which is unique up to scalar multiple).
Choose a sequence $(\varphi_n)_{n\in \IN}$ of $\ast$-automorphisms on $\mathcal{Z}_0 \otimes \IK$ and $(r_n)_{n\in \IN} \in (0, 1)^{\IN}$
satisfying
\[\tau \circ \varphi_n= r_n\cdot \tau \quad {\rm~for~each~}n\in \IN,\quad \sum_{n\in \IN} r_n \leq 1.\]
Take a free basis $(s_n)_{n\in \IN} \subset \IF_\infty$.
Let $\sigma \colon \mathbb{F}_\infty \acts \mathcal{Z}^{\otimes \mathbb{F}_\infty}$
denote the Bernoulli shift action.
For each $n \in \IN$, define 
\[\rho_n := {\rm R}_{g_n} \otimes \varphi_n \otimes \sigma_{s_n}.\]
Clearly each $\rho_n$ is a $\Gamma$-equivariant $\ast$-automorphism on $B$, and the family $\fr:=(\rho_n)_{n \in \IN}$ satisfies the conditions $(\diamondsuit 1)$ and $(\diamondsuit 2)$ in Section \ref{section:corr}.
Hence the \Cs-algebra $\T$ is simple by Theorem \ref{Thm:Tsimple}.
By Theorem \ref{Thm:ame}, the induced action $\Theta\colon \Gamma \acts \T$ is amenable.
By Theorem \ref{Thm:Ttrace} and Remark \ref{Rem:Ttrace}, the $\Gamma$-\Cs-algebra $\T$ has a unique $\Gamma$-invariant tracial weight up to scalar multiple,
since this is the case for $B$.

We next define $A:=\T \otimes \left(\bigotimes_{\IN} \mathcal{Z}^{\otimes \Gamma}\right) \otimes \mathcal{Z}_0$.
We equip $A$ with the $\Gamma$-action $\alpha:=\Theta \otimes \left(\bigotimes_{\IN} \sigma\right) \otimes \id_{\cZ_0}$.
Here $\sigma \colon \Gamma \acts \mathcal{Z}^{\otimes \Gamma}$ denotes the Bernoulli shift action.
Note that $A$ admits a $\Gamma$-invariant tracial weight $\tau_A$, which is unique up to scalar multiple.
It is not hard to check that $A \rtimes_\alpha \Gamma$ is simple (for instance, use Kishimoto's theorem \cite{Kis}).
It is easy to check that both  \Cs-algebras $A$ and $A \rtimes_\alpha \Gamma$ are classifiable, stable, simple, and projectionless.
By Theorem \ref{Thm:Kas}, $A$ is ${\rm KK}^\Gamma$-equivalent to $B\otimes \left(\bigotimes_{\IN} \mathcal{Z}^{\otimes \Gamma}\right)$.
Then, since $\left(B\otimes \left(\bigotimes_{\IN} \mathcal{Z}^{\otimes \Gamma} \right) \right) \rtimes_{\theta \otimes (\bigotimes_{\IN} \sigma)} \Gamma$ is stably isomorphic to $\mathcal{Z}_0$
(by the Green implimitivity theorem),
the Kasparov reduced crossed product functor \cite{Kas} gives the KK-equivalence between
$A\rtimes_\alpha \Gamma$ and $\cZ_0$.

Now, thanks to the classification theorem of simple stably projectionless \Cs-algebras \cite{GL}, Theorem 1.2,
to show the $\ast$-isomorphism $A\rtimes_\alpha \Gamma \cong \cZ_0\otimes \IK$,
we only need to check that the trace space of $A\rtimes_\alpha \Gamma$ is a ray.
To prove this, it suffices to show that any tracial weight $\tau$ on $A\rtimes_\alpha \Gamma$
factors through the canonical conditional expectation $E\colon A \rtimes_\alpha \Gamma \rightarrow A$: $\tau=\tau\circ E$.
To prove this equality, by Proposition \ref{Prop:trace}, it is enough to show the equality
\[\tau(a \lambda^{(\alpha)}_g)=0 \quad
{\rm~for~all~}a\in {\rm Ped}(\T) \odot \left(\bigodot_{\IN} \mathcal{Z}^{\odot \Gamma}\right) \odot {\rm Ped}(\cZ_0), g\in \Gamma \setminus \{e\}.\]
Here $\lambda^{(\alpha)}_g\in \cM(A\rtimes_\alpha \Gamma)$ denotes the canonical implementing unitary element of $g$,
and `$\odot$' stands for the algebraic tensor product.
By the linearity of $\tau$ and the definition of the Pedersen ideal, we may further assume that there exists $e\in {\rm Ped}(\T) \odot {\rm Ped}(\cZ_0)$
with $ae=ea=a$.
For any such $a$ and $g$, by the observation in Example 4.14 of \cite{SuzCMP},
one can choose a sequence
$b_1, \ldots, b_r\in \left(\bigotimes_{\IN} \mathcal{Z}^{\otimes \Gamma}\right)_+$
satisfying
\[\sum_{i=1}^r b_i^2=1,\quad \sum_{i=1}^r b_i a \alpha_g(b_i)=0.\]
This proves
\[\tau(a \lambda^{(\alpha)}_g)=\sum_{i=1}^r \tau(b_ie b_i a \lambda^{(\alpha)}_g)=\sum_{i=1}^r \tau(b_i a \alpha_g(b_i) \lambda^{(\alpha)}_g e)=0.\]
\end{proof}
The strategy of the proof in the UHF algebra case is essentially the same as the proof of Theorem \ref{Thm:dec},
and we only need to replace some ingredients therein.
We therefore only give a sketch of the proof.
We note that this in particular gives the first examples of amenable actions (of non-amenable groups) on simple AF-algebras.
\begin{Thm}
For any countable group $\Gamma$ and any UHF algebra $U$ of infinite type,
there is an amenable action $\alpha \colon \Gamma \acts A$
on a simple AF-algebra $A$
whose reduced crossed product $A\rtimes_{\alpha} \Gamma$ is
isomorphic to $U \otimes \IK$.
\end{Thm}
\begin{proof}
We define
\[C:= U \otimes \IK \otimes \mathcal{Z}^{\otimes \mathbb{F}_\infty}, \quad B:=c_0(\Gamma) \otimes C.\]
We equip $B$ with the left translation $\Gamma$-action $\theta:={\rm L}\otimes \id_{C}$.

Choose a sequence $(g_n)_{n\in \IN} \subset \Gamma$ which generates $\Gamma$ as a semigroup.
Fix a tracial weight $\tau$ on $U \otimes \IK$ (which is unique up to scalar multiple).
Take $N \in \IN \setminus \{0, 1\}$ which divides the supernatural number corresponding to $U$.
Choose a sequence $(\varphi_n)_{n\in \IN}$ of $\ast$-automorphisms on $U \otimes \IK$
satisfying
\[\tau_0 \circ \varphi_n= \frac{1}{N^{n+1}} \cdot \tau_0 \quad {\rm~for~each~}n\in \IN.\]
Take a free basis $(s_n)_{n\in \IN} \subset \IF_\infty$.
Let $\sigma \colon \mathbb{F}_\infty \acts \mathcal{Z}^{\otimes \mathbb{F}_\infty}$
denote the Bernoulli shift action.
Define
\[\fr:= ({\rm R}_{g_n} \otimes \varphi_n \otimes \sigma_{s_n})_{n\in \IN}.\]
Then, by the same proof as in Theorem \ref{Thm:dec}, it follows that the \Cs-algebra $A:=\T \otimes U$ is simple and classifiable,
and by the classification theorem of simple \Cs-algebras with a nontrivial projection (for the precise statement, see \cite{Win}, Theorem D), the resulting action
$\alpha \colon \Gamma \acts A$ satisfies
$A\rtimes_\alpha \Gamma\cong U \otimes \IK$.

Now it is remained to show that $A$ is an AF-algebra.
Note that $A\otimes U \cong A$.
By Theorems \ref{Thm:Ttrace} and \ref{Thm:Kas},
the projections of $A$ separate tracial weights on $A$ (because this is the case for $B$).
Then, R{\o}rdam's theorem \cite{Ror2}
implies that $A$ is of real rank zero.
Again by Theorem \ref{Thm:Kas}, K$_0(A)$ is torsion-free, and K$_1(A)=\{0\}$.
By Theorem 1.6 of \cite{Zha}, the ordered group K$_0(A)$ is a dimension group.
The classification theorem (see \cite{Win}, Theorem D)
shows that the \Cs-algebra $A$ is in fact an AF-algebra.
\end{proof}

\subsection*{Acknowledgements}
This work was supported by JSPS KAKENHI (Grant-in-Aid for Early-Career Scientists)
Grant Numbers JP19K14550, JP22K13924.

\end{document}